\RequirePackage{fix-cm}
\documentclass[smallextended]{svjour3}       
\smartqed  
\usepackage{graphicx}
 \usepackage{cite}
\usepackage{amsmath}
\usepackage{xcolor}
\usepackage{rotating,booktabs,multirow}
\usepackage{varwidth}

\newcommand{\argmin}{\operatornamewithlimits{argmin}}
\usepackage{amsfonts}
\usepackage{algorithm,algpseudocode}
\algnewcommand{\Inputs}[1]{%
  \State \textbf{Inputs:}
  \Statex \hspace*{\algorithmicindent}\parbox[t]{.8\linewidth}{\raggedright #1}
}
\algnewcommand{\Initialize}[1]{%
  \State \textbf{Initialization:}
  \Statex \hspace*{\algorithmicindent}\parbox[t]{.8\linewidth}{\raggedright #1}
}
\begin{document}

\title{A new dual for quadratic programming and its applications}
\subtitle{}


\author{Moslem Zamani       
}

\institute{M. Zamani \at
              Parametric MultiObjective Optimization Research Group, Ton Duc Thang University, Ho Chi Minh City, Vietnam\\
              Faculty of Mathematics and Statistics, Ton Duc Thang University, Ho Chi Minh City, Viet-nam\\
              School of Mathematics, Statistics and Computer Science, College of Science, University of Tehran, Enghelab Avenue, Tehran, Iran\\
                \email{zamani.moslem@tdt.edu.vn}           
}

\date{Received: date / Accepted: date}

\maketitle
\begin{abstract}
The main outcomes of the paper are divided into two parts. First, we present a new dual for quadratic programs,  in which, the dual variables are affine functions, and  we prove strong duality. Since the new dual is intractable, we consider a modified version by restricting the feasible set. This leads to a new bound for quadratic programs.  We demonstrate that the dual of the bound is a semi-definite relaxation of quadratic programs. In addition,  we probe the relationship between this bound and the well-known bounds. In the second part, thanks to the new bound, we propose a branch and cut algorithm for concave quadratic programs. We establish that the algorithm enjoys global convergence. The effectiveness of the method is illustrated for numerical problem instances.
\keywords{Non-convex quadratic programming \and Duality \and  Semi-definite relaxation \and Bound \and Branch and cut method \and Concave quadratic programming}
\end{abstract}

\section{Introduction}
\label{intro}

 We consider the following  quadratic program (QP):
\begin{equation}\tag{QP}\label{P}
\begin{array}{ll}
 \ & \min \ x^TQx+2c^Tx
\\ &  s.t. \ Ax\leq b,
\end{array}
\end{equation}
where $Q$ is a real symmetric $n\times n$ matrix, $A$ is a real $m\times n$ matrix, $c\in\mathbf{R}^n$ and $b\in\mathbf{R}^m$. Moreover, throughout the paper, it is assumed that  the feasible set, $X=\{x\in \mathbf{R}^n: Ax\leq b\}$, is nonempty and bounded. It is well-known that (QP) is solvable in polynomial time when $Q$ is positive semi-definite. Nevertheless,  indefinite QPs,  in which $Q$ is an indefinite matrix, are NP-hard even for rank-1 cases \cite{Par, Sahni}. In the paper, our focus is on non-convex QPs.\\
\indent
 Duality plays a fundamental role in optimization, from both theoretical and numerical points of view \cite{Luen}. It serves as a strong tool in stability and sensitivity analysis. For convex problems, duality is employed in some numerical methods for obtaining or verifying an optimal solution \cite{Boyd, Nem}.  \\
 \indent
It is well-known that the (Lagrangian) dual of a convex QP is also a convex QP, and satisfies strong duality. However, for non-convex case the dual of QPs might be meaningless because  the objective function of the dual problem might be  $-\infty$ while the primal has a finite optimal value.  \\
\indent
 The strong duality holds for convex QPs, though this property is still valid for some non-convex cases.  Optimizing a quadratic function on the sublevel set of a quadratic function is an archetype. S-lemma guarantees strong duality under Slater condition \cite{Ter}. \\ 
\indent
Global optimization methods for QPs are typically based on the convex relaxations and bounds. The most effective relaxations  for QPs (with quadratic constraints) rest upon the semidefinite programming and the reformulation-linearization technique (RLT) \cite{ Bao, Sherali2}. The semi-definite relaxations were first applied to some combinatorial problems \cite{Lov}. Due to their efficiency, these methods  have been extended to QPs with quadratic constraints \cite{Nest}. For more discussion on semidefinite relaxations and their comparisons, we refer the reader to the recent survey \cite{Bao}. Moreover, recently it has been shown that the combination of the semidefinite relaxations and RLT  leads to  stronger relaxations  \cite{Ans, Bao}.\\
\indent
 In addition to the relaxation methods, scholars have proposed some bounds for  classes of QPs \cite{Bomze1}. Similar to the  relaxation methods bounds give a lower bound. The most effective bounds for QPs are based on semidefinite programming  \cite{Bomze1}.\\
\indent
Another method which is able to give a bound for QPs is the so-called Lasserre hierarchy \cite{Las}. In fact, this method provides optimal value.  Lasserre hierarchy is able to tackle polynomial optimization problems (optimizing a polynomial function on a given  semi-algebraic set). It is well-known that  polytopes are Archimedean. Hence, due to the Putinar's Positivstellensatz theorem, optimal value of \eqref{P} is obtained by solving the following convex optimization problem:
\begin{equation}
\begin{array}{lll}
 \ & \max \ell
\\ &  s.t.  \  x^TQx+2c^Tx-\ell=\sigma_0(x)-\sum_{i=1}^{m} \sigma_i(x)(A_ix-b_i),\\
& \ \  \ \  \ \sigma_i\in \Sigma[x], \ i=0, 1, ..., m,
\end{array}
\end{equation}
 where $\Sigma[x]$ denotes the cone of polynomials which are sums of squares (SOS) \cite{Las}. By virtue of Lasserre hierarchy, the optimal value of \eqref{P} can be obtained by solving the finite number of semi-definite programs. However, the dimension of  semi-definite programs may increase dramatically \cite{Las}.\\
\indent
 Note that the Handelman's approximation hierarchy can be also used to produce a bound for QPs. Indeed, this method  provides optimal value under some mild conditions \cite{Lau}. In this approach, each subproblem is a linear program,  though similar to the Lasserre hierarchy  the dimension of  linear programs may increase exponentially. For more details on the method, we refer the interested reader to \cite{Lau, Las}.\\
 \indent
Concave QPs are important both theoretically and practically. Concave QPs appear in many applications including fixed charge and risk management problems and quadratic assignment problems \cite{Floudas, Buc, Klo}. In addition, it has been shown that some class of QPs can be reformulated as  concave QPs \cite{Floudas, Klo}.
It is well-known that a concave QP realizes its minimum at some vertices \cite{Floudas}. So, the problem is equivalent to the combinatorial problem of  optimizing a quadratic function on the vertices of a given polytope. This problem, as mentioned earlier, is NP-hard. Many avenues for tackling concave QPs have been pursued. Typical approaches are cutting plane methods, successive approximation methods and branch and bound approaches \cite{Horst}. For more methods and details, see also \cite{Floudas, Horst, Tuy}.  \\
\indent
One of the effective approach  for solving QPs is mixed integer programming reformulation \cite{Xiam}. As there exist  current state-of-the-art  mixed integer programming  solvers including GUROBI and CPLEX, this method may be very efficient.  Note that some solvers including CPLEX takes advantage of this idea to handle QPs \cite{Cplex}. Semidefinite relaxations have been also employed in branch and bound method for solving QPs \cite{Bur2, Chen}.  \\
\indent
Another approach which has deserve to be mentioned here is copositive programming method.  It is known that  a QP with quadratic constraints can be formulated as a linear program over the cone of completely positive matrices; See \cite{Bur} and references therein. Although the latter is convex, the cone of completely positive matrices is intractable. In fact, it is also an NP-hard problem. \\
 \indent
 The paper is organized as follows. After reviewing terminologies and notations, the new dual for QPs is introduced in Section 2. Dual variables are affine functions and strong duality is proved. As the dual problem is intractable, we take into account a subset of the feasible set, leading to a new bound for QPs. We show that the bound is well-defined for QPs with bounded feasible set. Moreover, we establish that the bound is invariant under affine transformation, and is independent of the algebraic representation of $X$. \\
 \indent 
  In Section 3, we investigate the relationship between the new bound and the conventional bounds. We prove that the new bound is equivalent to the best proposed bound  for standard QPs. Moreover, we show that for box constrained QPs the dual of the new bound is Shor relaxation with partial first-level RLT. \\
  \indent
  Section 4 is devoted to concave QPs. Thanks to the new bound, we introduce a new branch and cut method. In Section 5, we illustrate the effectiveness of the proposed method by presenting its numerical performance on some concave QPs.
\subsection{Notation}
 The $n$-dimensional Euclidean
space is denoted by $\mathbf{R}^n$.  We denote the  $i^{th}$ row of a given matrix $A$ by $A_i$. Vectors are considered to be column vectors and the superscript $T$ denotes the transpose operation. W use $e$ and and $e_i$ to denote vector of ones and $i^{th}$ unit vector, respectively. 
The nonnegative orthant is denoted by $\mathbf{R}_{+}^n$. Notation $A \succeq B$ means  matrix $A-B$ is positive semidefinite. Furthermore, $ A \bullet B$ denotes the inner product of A and B,  i.e., $A \bullet B=trace(AB^T)$. \\
\indent
For a set $X\subseteq\mathbf{R}^n$, we use the notations $int(X)$ and
$cone(X)$ for the interior and the convex conic hull of  $X$, respectively. For a convex cone $K$, its dual cone is defined and denoted by $K^*=\{y: y^Tx\geq 0, \ \forall x\in K\}$.
 Two notations $\nabla f(\bar x)$ and $\nabla^2 f(\bar x)$ stand for the gradient and Hessian of smooth function $f$ at $\bar x$. For the affine function 
 $\alpha: \mathbf{R}^n \to  \mathbf{R}$ given by 
 $\alpha(x)=a^Tx+a_0$, its norm is defined and denoted by $\|\alpha\|=\max_{i=0,1, ..., n} |a_i|$.
\section{A new dual for quadratic programs}
\label{sec:1}
 In this section, we present a new dual for QPs. Throughout the section, it is assumed that $X$ is a bounded polyhedral.
We propose the following convex optimization problem as a dual of  \eqref{P},
\begin{equation}\label{Dn}
\begin{array}{ll}
 \ & \max \ \ell \\
 \  &  s.t.\   x^TQx+2c^Tx-\ell+\sum_{i=1}^{m} \alpha_i(x)(A_ix-b_i)\in P[x],\\
\  & \ \ \ \  \  \sum_{i=1}^{m} \alpha_i(x)(A_ix-b_i)\leq 0, \  \  x\in X,
\end{array}
\end{equation}
where $\alpha_i$, $i=1, ..., m$, are affine functions and $P[x]$ denotes the set of nonnegative polynomials on $\mathbf{R}^n$. It is readily seen that the above problem is a convex problem with  infinite constraints. Note that a quadratic function $q(x)=x^TQx+2c^Tx+c_0$ is nonnegative on $\mathbf{R}^n$ if and only if 
$
 \begin{pmatrix}
Q & c \\ c & c_0
\end{pmatrix}\succeq 0
$.
We prove that problem \eqref{Dn} is feasible and fulfills strong duality.  Before we get to the proof, we need to present a lemma.

\begin{lemma}\label{L1n}
Let $X=\{x\in\mathbf{R}^n: Ax\leq b\}$ be a full-dimensional polytope and let $q(x)=x^TQx+2c^Tx+c_0$ be a quadratic function. Then there exist affine functions $\alpha_i$ for $i=1, ..., m$ such that
$$
x^TQx+2c^Tx+c_0=\sum_{i=1}^{m} \alpha_i(x)(A_ix-b_i).
$$
\begin{proof}
The existence of affine functions $\alpha_i$, $i=1, ..., m$, satisfying the desired equality is equivalent to the consistency of the linear system
 $$
\frac{1}{2} \sum_{i=1}^{m}
 \begin{pmatrix}
d_i \\ f_i
\end{pmatrix}
 \begin{pmatrix}
A_i & -b_i
\end{pmatrix}
+
\frac{1}{2} \sum_{i=1}^{m}
 \begin{pmatrix}
A_i^T \\ -b_i
\end{pmatrix}
 \begin{pmatrix}
d_i^T & f_i
\end{pmatrix}=
 \begin{pmatrix}
Q & c \\ c & c_0
\end{pmatrix},
$$
in which $ \begin{pmatrix}
d_i \\ f_i
\end{pmatrix}$
, $i=1, ..., m$, are variables. Indeed,  $ \begin{pmatrix}
d_i \\ f_i
\end{pmatrix}$
is the representative  of $\alpha_i$. To prove the consistency, it is sufficient to show that the above system has full rank. On the contrary, suppose that the above linear system does not have full rank. So, there exists a non-zero symmetric matrix $D$ such that
 $$
\frac{1}{2} \sum_{i=1}^{m}
D\bullet \begin{pmatrix}
d_i \\ f_i
\end{pmatrix}
 \begin{pmatrix}
A_i & -b_i
\end{pmatrix}
+
\frac{1}{2} \sum_{i=1}^{m}
 D\bullet \begin{pmatrix}
A_i^T \\ -b_i
\end{pmatrix}
 \begin{pmatrix}
d_i^T & f_i
\end{pmatrix}=
0
$$
Since $A\bullet x^Ty=x^TAy$, we have
 $$
 \sum_{i=1}^{m}
 \begin{pmatrix}
A_i & -b_i
\end{pmatrix}
D \begin{pmatrix}
d_i \\ f_i
\end{pmatrix}
=
0.
$$
As  $\Big\{\begin{pmatrix}
A_1^T \\ -b_1
\end{pmatrix}, ..., \begin{pmatrix}
A_m^T \\ -b_m
\end{pmatrix}\Big\}$
generates $\mathbf{R}^{n+1}$, $D$ must be zero. This contradicts our assumption $D\neq 0$ and implies the consistency of the linear system.
\end{proof}
\end{lemma}

\indent
The following theorem shows that the proposed dual satisfies strong duality.
\begin{theorem} \label{SD}
Let $X$ be a polytope. The optimal values of problems \eqref{P} and \eqref{Dn} are equal.
\begin{proof}
Without loss of generality, we may assume that $X$ is full-dimensional. Otherwise it is enough to consider \eqref{P} on the affine space  generated by $X$. Suppose that $q^{\star}$ is the optimal value of \eqref{P}. By Lemma \ref{L1n}, there exists $\bar \alpha_i$, $i=1, ..., m$, such that 
$$
x^TQx+2b^Tx-q^{\star}=-\sum_{i=1}^{m} \bar\alpha_i(x)(A_ix-b_i).
$$
So,  $\bar \alpha_i$, $i=1, ..., m$, and $\bar\ell=q^{\star}$ are feasible for \eqref{Dn} and the optimal value of problem \eqref{Dn} is greater than or equal to $q^{\star}$. The  constraints of problem \eqref{Dn} imply that the optimal value of the dual problem is not greater than $q^{\star}$. So, the aforementioned feasible point is optimal solution for  the dual problem and the proof is complete.
\end{proof}
\end{theorem}

\indent
Although problem \eqref{L1n} is convex, it is not tractable. This difficulty is caused by the number of constraints. So, to take advantages of this formulation, we need to adopt a procedure which one is able to handle  the dual problem. A natural approach is the maximization of the objective function on a subset of the feasible set, i.e. restricting the feasible set of the dual problem. In the rest of the paper, we restrict the variables ,$\alpha_i$, to nonnegative affine functions on $X$.  \\
\indent 
If the vertices of $X$ are also available one can consider the following set which includes the aforementioned set. Let $v_1, ..., v_k$ denote the vertices of $X$. It is easily seen that the affine functions $\alpha_i$, $i=1, ..., m$, which fulfill the following inequalities are feasible for problem \eqref{Dn},
\begin{align}\label{gfgh}
\sum_{i=1}^m(A_iv_j-b_i)\alpha_i(x)\leq 0, \  x\in X, j=1, ..., k.
\end{align}
Non-homogenous  Farkas' Lemma provides an explicit form of the affine functions which satisfy \eqref{gfgh}. So, it can be formulated by a finite number of linear inequalities. If $cone(\{Ax-b: x\in X\})=-\mathbf{R}^m_+$, then  $\alpha_i$, $i=1, ..., m$, that fulfill \eqref{gfgh} are nonnegative affine functions on $X$.\\
\indent
 Let $\mathcal{A}_+(X)$ denote the set of nonnegative affine functions on non-empty polytope $X$. By non-homogenous  Farkas' Lemma, 
 $\alpha(x)=d^Tx+f$ is nonnegative on nonempty polyhedron $X$ if and only if there exist nonnegative scalars $\lambda_i$, $i=0, ..., m$, with
 $\alpha(x)=\lambda_0+\sum_{i=1}^m \lambda_i(b_i-A_ix)$; See Theorem 8.4.2 in \cite{Mang}. It is easily seen that  $\mathcal{A}_+(X)$ is a polyhedral  cone with nonempty interior  \cite{Mang}. To tackle problem \eqref{Dn}, we consider the following restricted problem,
 \begin{equation}\label{L}
\begin{array}{ll}
 \ & \max \ \ell \\
 \  &  s.t.\   x^TQx+2c^Tx-\ell+\sum_{i=1}^{m} \alpha_i(x)(A_ix-b_i)\in P[x],\\
\  & \ \ \ \ \ \  \alpha_i\in \mathcal{A}_+(X), \ i=1, ..., m.
\end{array}
\end{equation}
\indent
 Considering affine functions instead of scalars for Lagrange multipliers can be found in the literature. To extend  S-lemma, Sturm et al. applied an affine function as a Lagrange multiplier \cite{Sturm}. To make the point clear, consider the optimization problem
  \begin{equation}\label{St}
\begin{array}{ll}
 \ & \min \ x^TQ_1x+2c_1^Tx \\
 \  &  s.t.\   x^TQ_2x+2c_2^Tx+b_2\leq 0,\\
\  & \ \  \ \ \ \ c_3^Tx+b_3\leq 0,
\end{array}
\end{equation}
 where $Q_2 \succeq 0$ and $\mathcal{X}=\{x: x^TQ_2x+2c_2^T x+b_2\leq 0, c_3^T\bar x+b_3\leq0\}$ has nonempty interior.  They show that the optimal values of problem \eqref{St} and the following problem are the same,
  \begin{equation*}
\begin{array}{ll}
 \ & \max \ \ell \\
 \  &  s.t.\   x^TQ_1x+2c_1^Tx-\ell+t(x^TQ_2x+2c_2^Tx+b_2)+(e^Tx+f)(c_3^Tx+b_3)\in P[x],\\
\  & \ \ \ \ \ \  t\geq 0, \ \begin{pmatrix} d\\ f \end{pmatrix}\in \mathcal{H}^*,
\end{array}
\end{equation*}
 where the convex cone $\mathcal{H}=\Big\{\begin{pmatrix} x\\ x_0 \end{pmatrix}: x^TQ_2x+2x_0c_2^Tx+b_2x_0^2\leq 0, 2c_2^Tx+b_2x_0\leq 0, x_0\geq 0\Big\}$. As seen, the Lagrange multiplier of the linear constraint is  an affine function. Unlike problem \eqref{L}, which the feasible affine functions are characterized by the feasible set,  only the quadratic constraint determines a feasible affine function. It may be of interest to verify the validity of Sturm et al.'s result when the constraint 
 $\begin{pmatrix} d\\ f \end{pmatrix}\in \mathcal{H}^*$
 is replaced with $d^Tx+f \in \mathcal{A}_+(\mathcal{X})$. The following proposition says that the result also holds under these conditions.
 
 \begin{proposition}\label{PSt}
If $Q_2 \succeq 0$ and $int(\mathcal{X}) \neq \emptyset$, then the optimal value of the following problem is equal to that of problem \eqref{St},
  \begin{align}\label{Stz}
\nonumber \ & \max \ \ell \\
\nonumber  \  &  s.t.\   x^TQ_1x+2c_1^Tx-\ell+t(x^TQ_2x+2c_2^Tx+b_2)+(d^Tx+f)(c_3^Tx+b_3)\in P[x],\\
\  & \ \ \ \ \ \  t\geq 0, \  d^Tx+f \in \mathcal{A}_+(\mathcal{X}).
\end{align}
\begin{proof}
It is seen that the optimal value of problem \eqref{Stz} is less than or equal to that of problem  \eqref{St}. By Sturm et al.'s result, to show equality, it suffices to prove the inclusion
 $\Big\{d^Tx+f: \begin{pmatrix}  d\\  f \end{pmatrix}\in \mathcal{H}^*\Big\}\subseteq\mathcal{A}_+(\mathcal{X})$. 
 Assume on the contrary, there exists
  $\begin{pmatrix} \bar d\\ \bar f \end{pmatrix}\in  \mathcal{H}^*$,
   but $\bar d^Tx+\bar f \notin \mathcal{A}_+(\mathcal{X})$. So, there exists $\bar x\in \mathcal{X}$ with $\bar d^T\bar x+\bar f <0$. The semi-positive definiteness of $Q_2$ and $\bar x\in \mathcal{X}$ imply that   
   $2c_2^T\bar x+b_2\leq -\bar x^TQ_2\bar x\leq 0$. Therefore, $\begin{pmatrix} \bar x\\ 1\end{pmatrix}\in  \mathcal{H}$. Since $\begin{pmatrix} \bar d\\ \bar f \end{pmatrix}\in  \mathcal{H}^*$, we must have $\bar d^T\bar x+\bar f \geq 0$, which contradicts the assumption that  $\bar d^T\bar x+\bar f < 0$ and completes the proof. 
\end{proof}
\end{proposition}

 \indent
 Note that under the assumptions of Proposition \ref{PSt} affine function $\alpha$ is nonnegative on $\mathcal{X}$ if and only if there exist $\lambda_1\geq 0$ and 
 $\lambda_2\geq 0$ with
 $$
 \alpha(x)+\lambda_1(x^TQ_2x+2c_2^Tx+b_2)+\lambda_2(c_3^Tx+b_3)\in P[x];
 $$
 See \cite{Mang} for proof.  Thus, problem \eqref{Stz} can be formulated as a semi-definite program.\\
 \indent 
  Typically, SOS polynomials are used as Lagrange multipliers in polynomial optimization  \cite{Lass, Kim, Chu}. However, to the best knowledge of author, affine functions have not been applied in this manner as dual variables in finite optimization theory.  \\
\indent
Let us return back to problem \eqref{L}. When $X$ is nonempty, \eqref{L} can be formulated as the semi-definite program:
\begin{equation}\label{ggf}
\begin{array}{ll}
 \ & \max \ \ell \\
 \  &  s.t.\   \frac{1}{2}\sum_{i=1}^n   \begin{pmatrix}
  d_i \\
f_i
\end{pmatrix}
  \begin{pmatrix}
  A_i  & -b_i
\end{pmatrix}+
 \frac{1}{2}\sum_{i=1}^n   \begin{pmatrix}
  A_i ^T \\
-b_i
\end{pmatrix}
  \begin{pmatrix}
  d_i^T  & f_i
\end{pmatrix}+
\begin{pmatrix}
Q & c\\
c^T &  -\ell
\end{pmatrix}
\succeq 0,
\\
\  & \ \  \  \  \begin{pmatrix}
d_i \\
f_i
\end{pmatrix}
\in cone
\Big\{
 \begin{pmatrix}
  -A_j^T\\
b_j
\end{pmatrix}, 1\leq j \leq m,
  \begin{pmatrix}
  0 \\
1
\end{pmatrix}
\Big \}, \ i=1, ..., m.
\end{array}
\end{equation}
Note that if $X$ is empty, the second constraint of \eqref{ggf} is not necessarily equivalent to the constraint $\alpha_i\in \mathcal{A}_+(X)$, $i=1, ..., m$. As the second constraint of \eqref{ggf}  gives $\alpha_i$ explicitly, the above problem can be written as 
\begin{align} \label{L00}
\nonumber \max \  \ & \ell \\
 \   s.t.\  \ &
 \frac{1}{2}
 \begin{pmatrix}
-A^TY A-A^T Y^TA &  \ A^T(Y b+Y^Tb+y)\\
(Y b+Y^Tb+y)^TA &   \ -2y^Tb-2b ^TY b
\end{pmatrix}
+\begin{pmatrix}
Q & c\\
c^T &  -\ell
\end{pmatrix}
\succeq 0,\\
\nonumber \ & Y \geq 0, \ y\geq 0,
\end{align}
where $Y\in\mathbf{R}^{m\times m}$ and $y\in\mathbf{R}^m$. In the above semi-definite-program, the matrix 
$$
\Gamma_{ij}=\frac{-1}{2}\begin{pmatrix}
A_i^TA_j+ A_j^TA_i & -b_iA_j^T-b_jA_i^T\\
-b_iA_j-b_jA_i &  2b_ib_j
\end{pmatrix}
$$ 
is the coefficient of $y_{ij}$. As the coefficients of $y_{ij}$ and $y_{ji}$ are the same, we can assume that $Y$ is symmetric. So, problem \eqref{L00} is reformulated as 
\begin{align} \label{L0}
\nonumber \max \  \ & \ell \\
 \   s.t.\  \ &
 \begin{pmatrix}
-A^TY A &  \ A^T(Y b+0.5y)\\
(Y b+0.5y)^TA &   \ -y^Tb-b ^TY b
\end{pmatrix}
+\begin{pmatrix}
Q & c\\
c^T &  -\ell
\end{pmatrix}
\succeq 0,\\
\nonumber \ & Y \geq 0, \ y\geq 0, \ Y^T=Y.
\end{align}

\indent
The next lemma shows that, under the boundedness of $X$, problem \eqref{L} is feasible.
\begin{proposition} \label{well}
If $X$ is a polytope, then problem \eqref{L} is feasible.
\begin{proof}
As $X$ is bounded, there exist scalars $f_i$, $i=1, ...,m$, such that $\alpha_i(x)=A_i x+f_i$ is positive on $X$ and $\alpha_i\in \mathcal{A}_+(X)$. Due to the boundedness of $X$, the system $Ad\leq 0$ does not have any non-zero solution. Thus, $H=\sum_{i=1}^{m} A_i^TA_i$ is positive-definite. By choosing $\gamma$ sufficiently large  and a suitable choice of $\ell$,  $\gamma\alpha_i$, $i=1, ...,m$, accompanying $\ell$ satisfy all constraints of problem \eqref{L}.
\end{proof}
\end{proposition}

\indent
Proposition \ref{well} does not hold necessarily for QPs with unbounded feasible set. The following example illustrates the point.
\begin{example}
Consider the QP
\begin{align*}
\nonumber \min \  \ & x_1x_2 \\
 \   s.t.\  \ &  x_1=0.
\end{align*}
In this example,  problem \eqref{L} is infeasible while the optimal value is zero.
\end{example}
\indent

 Problem \eqref{L} can be also interpreted  via S-lemma. Let $int(X)\neq \emptyset$. For given $\alpha_i\in \mathcal{A}_+(X)$, $i=1, ..., m$, we have $X\subseteq \{x: \sum_{i=1}^{m} \alpha_i(x)(A_ix-b_i)\leq 0\}$, which is an overestimation of $X$ with a sublevel set of a quadratic function. By S-lemma, the optimal value of $x^TQx+2c^Tx$ on this set is obtained by solving the problem 
 $$
 \{\max \ell: x^TQx+2c^Tx-\ell+t(\sum_{i=1}^{m} \alpha_i(x)(A_ix-b_i))\in P[x], \ t\geq 0\}.
 $$
 As $\mathcal{A}_+(X)$ is a cone, the first constraint of problem \eqref{L} is merely S-lemma and  problem \eqref{L} provides the greatest optimal value $x^TQx+2c^Tx$ on the sublevel sets of quadratic functions which include $X$. \\
\indent
In the sequel, we denote the set of $\{\sum_{i=1}^{m} \alpha_i(x)(A_ix-b_i):  \alpha_i(x)\in \mathcal{A}_+(X)\}$ by $\mathcal{N}(X)$. One can infer from the proof of Proposition \ref{well} that $\mathcal{N}(X)$ is nonempty and involves a strictly convex function. It is easily seen that $\mathcal{N}(X)$ is a closed, convex cone with nonempty interior. Moreover, if the interior of $X$ is nonempty, then $\mathcal{N}(X)$ is pointed, i.e. $\mathcal{N}(X)\cap-\mathcal{N}(X)=\{0\}$.\\
\indent
The cone $\mathcal{N}(X)$ can be applied for obtaining a convex underestimator of  a quadratic function on a given polytope. For \eqref{P}, one can consider $F: X\to \mathbf{R}$, defined by 
 $$
F(x)=\{\max \ x^TQx+2c^Tx+p(x): p\in \mathcal{N}(X), \nabla^2p+2Q\succeq 0\},
$$
as a convex underestimator.
Another problem in which $\mathcal{N}(X)$ might be useful is the approximation of  L\"{o}wner-John ellipsoid  for a given polytope. We refer the interested reader to \cite{Boyd, Nem}  for more details about L\"{o}wner-John ellipsoid.\\
\indent
As we consider a subset of the feasible set of dual problem, the optimal value of  \eqref{L} may be strictly less than that of \eqref{Dn} or equivalently \eqref{P}. So, we can regard problem  \eqref{L}  as a new bound for QPs. The following proposition states this fact.
\begin{proposition} \label{w}
Let $\bar x$ and $\bar{\alpha_i}$ $(i=1, ..., m)$, $\bar\ell $ be feasible for   \eqref{P} and problem \eqref{L}, respectively. Then $\bar x^TQ\bar x+2c^T\bar x\geq \bar \ell$.
\end{proposition}
\begin{proof}
The first constraint of problem \eqref{L} implies $\bar x^TQ\bar x+2c^T\bar x-\bar\ell+\sum_{i=1}^{m} \bar\alpha_i(\bar x)(A_i\bar x-b_i)\geq 0$. 
Since $-\sum_{i=1}^{m} \bar\alpha_i(\bar x)(A_i\bar x-b_i)\geq 0$, we get $\bar x^TQ\bar x+2c^T\bar x\geq \bar \ell$, which is the desired inequality.
\end{proof}

\indent
The next theorem provides  sufficient conditions under which bound \eqref{L} is exact. Let $I(x)$ denote active constraints at $x$, i.e. $I(x)=\{i: A_ix=b_i\}$.
\begin{theorem}\label{hgg}
Let $X$ be a polytope. The optimal values of \eqref{P} and problem \eqref{L} are the same if  there exist $\bar x\in argmin_{x\in X} x^TQx+2c^Tx$ and $d_i^Tx+f_i\in \mathcal{A}_+(X)$ for $i=1, 2, ..., m$ such that
$$
Q+\frac{1}{2}\sum_{i=1}^md_iA_i+\frac{1}{2}\sum_{i=1}^mA_i^Td_i^T\succeq 0,
$$
$$
Q\bar x+c+\sum_{i\in I(\bar x)}(d_i^T\bar x +f_i) A_i^T +\sum_{i\in \{1, 2, ..., m\}\setminus I(\bar x)}(A_i\bar x -b_i) d_i =0,
$$
and
$$
d_i^T\bar x +f_i=0, \ \forall i\in \{1, 2, ..., m\}\setminus I(\bar x).
$$
\begin{proof}
Consider the quadratic function $q(x)=x^TQx+2c^Tx+\sum_{i=1}^{m} (d_i^Tx+f_i)(A_ix-b_i)$. The first condition implies that $q$ is convex.  We can infer from the second and third conditions that $\nabla q(\bar x)=0$. By the sufficient optimality conditions for convex quadratic functions,  $\bar x$ is an optimal solution problem $\{\min q(x): x\in \mathbf{R}^n\}$ with the optimal value $\bar x^TQ\bar x+2c^T\bar x$. Additionally, $\alpha_i(x)=d_i^Tx+f_i$, $i=1, 2, ..., m$, and $\ell=\bar x^TQ\bar x+2c^T\bar x$ are feasible for  problem \eqref{L}. By virtue of  Proposition  \ref{w}, the optimal value of both problems are equal and the proof is complete.
\end{proof}
\end{theorem}

\indent 
We notice that checking the sufficient conditions provided in  Theorem \ref{hgg} is not difficult. In fact, it can be done by testing the feasibility of a semi-definite program, which there are polynomial time algorithms to do it.  Under convexity, the above theorem holds for any optimal solution, so the bound is exact in the convex case. However, in general, the assumptions in Theorem \ref{hgg} may not hold.  \\
\indent 
Another point concerning Theorem \ref{hgg} is that it provides sufficient conditions for global optimality. Some sufficient global optimality conditions by virtue of semi-definite relaxation can be found in the literature, see e.g., \cite{Zheng} and the references therein. As we will see in the sequel the dual of  problem \eqref{L} is a semi-definite relaxation of  \eqref{P}. So, most of the derived results with a little modification can be applied to problem \eqref{L}. We refer the interested reader to 
\cite{Bomze2, Floudas } for some necessary and sufficient global optimality conditions for QPs.\\
\indent 
In general, when the feasible set of a semi-definite program is unbounded, it is probable that it does not realize its optimal value \cite{Nem}. In the following proposition, we prove that problem \eqref{L} achieves its optimal value while the feasible set of \eqref{L} may be unbounded.
\begin{proposition} \label{popt}
If $X$ be a polytope, then problem \eqref{L} achieves its optimal value.
\begin{proof}
Without loss of generality, we can assume $int(X)\neq\emptyset$. Otherwise, it is enough to consider the polytope on the affine space generated by itself.  \\
By Proposition \ref{w}, the optimal value of  problem \eqref{L}, denoted by $\bar\ell$, is finite.  There exist sequences $\{\alpha_i^k\}\subseteq \mathcal{A}_+(X)$,$i=1, ..., m$,  and $\{\ell^k\}$
such that $\ell^k\to \bar\ell$ and $x^TQx+2c^Tx-\ell^k+\sum_{i=1}^{m} \alpha_i^k(x)(A_ix-b_i)\in P[x]$. If the sequences $\{\alpha_i^k\}$, $i=1, ..., m$, are bounded, then due to the closedness of $P[x]$ and $\mathcal{A}_+(X)$, the proof is complete. Otherwise, by setting $\mu_i^k=\frac{\alpha_i^k}{t_k}$ for $i=1, ..., m$ with $t_k=\max_{1\leq i\leq m}\|\alpha_i^k\|$, and choosing appropriate subsequences if necessary, we can assume that $\mu_i^k\to\bar\mu_i$ for $i=1, ..., m$ and there exists $j$ such that 
$\bar\mu_j \neq 0$. Due to the closedness of the set of nonnegative polynomials, we have $q(x)=\sum_{i=1}^{m} \bar\mu_i(x)(A_ix-b_i)\in P[x]$. As $int(X)\neq \emptyset$, there exists some $\bar x\in X$ such that $\bar\mu_j(\bar x)>0$ and $A_j\bar x<b_i$. Therefore,  $q(\bar x)<0$, which contradicts $q\in P[x]$.
\end{proof}
\end{proposition}

\indent
 It follows from   the proof of Proposition \ref{popt}  that all optimal solutions of  problem \eqref{L} are bounded when $int(X)\neq \emptyset$. One important question may arise about \eqref{L} is that: "Is the optimal value of problem \eqref{L} independent of the representation of $X$?". Next theorem gives the affirmative answer to the question.
\begin{theorem} \label{ind}
Let $\{x\in \mathbf{R}^n: \bar{A}x\leq \bar{b}\}$ and $\{x\in \mathbf{R}^n: \hat{A}x\leq \hat{b}\}$ be two different  representations of the polytope $X$. Then the optimal values of problem \eqref{L} corresponding to these representations are equal.
\begin{proof}
Suppose that $\hat\ell$ and $\bar\ell$ are the optimal values of problem \eqref{L} corresponding to the representations $\{x\in \mathbf{R}^n: \hat{A}x\leq \hat{b}\}$ and 
$\{x\in \mathbf{R}^n: \bar{A}x\leq \bar{b}\}$, respectively, with $\hat A \in \mathbf{R}^{\hat m\times n}$ and  $\bar A \in \mathbf{R}^{\bar m\times n}$. By Proposition \ref{popt}, there exist $\hat{\alpha_i}$ $(i=1, ..., \hat m)$ which are optimal for
\begin{equation}\label{Ln}
\begin{array}{ll}
 \  & \max \ell\\
 \  & s.t.  \   x^TQx+2c^Tx -\ell+\sum_{i=1}^{\hat m} \alpha_i(x)(\hat{A_i}x-\hat{b_i})\in P[x],\\
\  & \ \  \ \ \ \alpha_i\in \mathcal{A}_+(X), \ i=1, ..., \hat m.
\end{array}
\end{equation}
As  $\hat{b_i}-\hat{A_i}x\in \mathcal{A}_+(X)$, there exist nonnegative scalars $Y_{ij}$, 
$j=0, 1, ..., \bar m$, such that
$$
\hat{A_i}x-\hat{b_i}=\sum_{j=1}^{\bar m} Y_{ij}( \bar{A_j}x-\bar{b_j})- Y_{i0}.
$$
Similarly, according to $\hat{\alpha_i}\in \mathcal{A}_+(X)$, there exist nonnegative scalars $W_{ij}$, $j=0, 1, ..., \bar m$, with
$\hat{\alpha_i}(x)=\sum_{j=1}^{\bar m} -W_{ij}( \bar{A_j}x-\bar{b_j})+ W_{i0}$.
By the first constraint of problem \eqref{Ln}, we get
\begin{align*}
& x^TQx+2c^Tx-\hat{\ell}+\sum_{i=1}^{\hat m} \hat{\alpha_i}(x)\Big(\sum_{j=1}^{\bar m}Y_{ij}( \bar{A_j}x-\bar{b_j})- Y_{i0}\Big)=\\
& x^TQx+2c^Tx-\hat{\ell}+\sum_{j=1}^{\bar m} ( \bar{A_j}x-\bar{b_j})\Big(\sum_{i=1}^{\hat m} Y_{ij}\hat{\alpha_i}(x)\Big)-\sum_{i=1}^{\hat m}  Y_{i0}\hat{\alpha_i}(x)=\\
& x^TQx+2c^Tx-(\hat{\ell}+\sum_{i=1}^{\hat m}  Y_{i0}W_{i0})+\sum_{j=1}^{\bar m} ( \bar{A_j}x-\bar{b_j})\Big(\sum_{i=1}^{\hat m} (Y_{ij}\hat{\alpha_i}(x)+W_{ij}Y_{i0})\Big)\\ & \in P[x].
\end{align*}
As $\mathcal{A}_+(X)$ is a convex cone,  $\tilde \alpha_i(x)=\sum_{i=1}^{\hat m} (Y_{ij}\hat{\alpha_i}(x)+W_{ij}Y_{i0})\in \mathcal{A}_+(X)$ for $i=1, ..., \bar m$. So, $\tilde \ell=\hat{\ell}+\sum_{i=1}^{\hat m}  Y_{i0}W_{i0}$ and $\tilde \alpha_i$, $i=1, ..., \bar m$, are feasible for problem \eqref{L} corresponding to the representation
$X=\{x\in \mathbf{R}^n: \bar{A}x\leq \bar{b}\}$. This implies $\bar\ell\geq \hat\ell$, because $Y_{i0}W_{i0}\geq 0$.  
 By a similar argument, one can establish $\hat\ell\geq \bar\ell$. Therefore,  $\bar\ell= \hat\ell$ and the proof is complete.
\end{proof}
\end{theorem}

\indent
Similar to the  proof of Theorem \ref{ind}, one can prove that if polytope $X_1$ is a subset of polytope $X_2$, then $opt(X_1, Q, c)\leq opt(X_2, Q, c)$. (Let $opt(X, Q, c)$ denote  the optimal value of problem \eqref{L} corresponding to polytope $X$, matrix $Q$ and vector $c$). We call this property as inclusion property. In the following proposition, we show that the optimal value of problem \eqref{L} is constant under invertible affine transformation.
\begin{proposition} \label{tra}
The optimal value of problem \eqref{L} is invariant under affine invertible transformations.
\begin{proof}
Let $T$ be an invertible affine transformation on $\mathbf{R}^n$. We set $Y=T^{-1}(X)$ and assume that $T(y)=Hy+h$ for some invertible matrix $H$ and $h\in\mathbf{R}^n$. For $\alpha_i\in \mathcal{A}_+(X)$, $i=1, ..., m$, and $\ell$ satisfying 
$$
 x^TQx+2c^Tx -\ell+\sum_{i=1}^{m} \alpha_i(x)(A_ix-b_i)\in P[x],
$$
we have
$$
y^T\bar Qy+2\bar c^Ty+c_0 -\ell+\sum_{i=1}^{m} \mu_i(y)(\bar A_iy-\bar b_i)\in P[y],
$$
where $\mu_i=\alpha_iT\in \mathcal{A}_+(Y)$, $\bar Q=H^TQH$, $\bar c=Hc+H^TQh$, $c_0=h^TQh+2c^Th$, $\bar A=AH$ and $\bar b=b-Hh$. This statement implies  
$opt(X, Q, c) \leq opt(Y, \bar Q, \bar c) -c_0$. By the same argument, one can derive that  $opt(Y, \bar Q, \bar c) -c_0  \leq opt(X, Q, c)$. This completes the proof.
\end{proof}
\end{proposition}

\indent
The following proposition states if a polytope is singleton, then bound \eqref{L} is exact.
\begin{proposition} \label{sing}
If $X=\{\bar x\}$, then the optimal values of \eqref{P} and \eqref{L} are equal.
\begin{proof}
According to  Proposition \ref{tra}, the optimal value of problem \eqref{L} is independent of translation. So,  without loss of generality $\bar x=0$. In addition, by virtue of  Theorem \ref{ind}, we assume $X=\{x: x=0\}$. We define $\alpha_i, \mu_i\in \mathcal{A}_+(X)$ for $i=1, ..., n$,  corresponding to  \eqref{P}, as follows
\[
\alpha_i(x)=\begin{cases}
      0 & c_i\geq 0 \\
      (-Qx)_i-2c_i &  c_i < 0 \\
   \end{cases}
  \quad \mu_i(x)=\begin{cases}
      (Qx)_i+2c_i & c_i\geq 0 \\
      0 & c_i <  0 \\
   \end{cases}
\]
 Thus, $x^TQx+2c^Tx +\sum_{i=1}^{m} \alpha_i(x)(x_i)+\sum_{i=1}^{m} \mu_i(x)(-x_i)=0$. On account of Proposition \ref{w}, the optimal value of \eqref{L} is zero, which is the desired conclusion.
\end{proof}
\end{proposition}

\indent
Since \eqref{L} is a convex optimization problem, it is natural to ask about its dual.  To get the dual of problem \eqref{L}, we consider problem \eqref{L0}.  The dual of problem \eqref{L0} is formulated as follows,
\begin{align}\label{DD}
\nonumber \  \min \ &
  \begin{pmatrix}
Q & c\\
c^T &  0
\end{pmatrix}
\bullet
\begin{pmatrix}
X & x\\
x^T &  x_0
\end{pmatrix}
\\
 \    s.t.\   & \  
  \begin{pmatrix}
A_i^TA_j & -b_jA_i^T\\
-b_jA_i &  -b_ib_j
\end{pmatrix}
\bullet
\begin{pmatrix}
X & x\\
x^T &  x_0
\end{pmatrix}\geq 0, \  i\leq j= 1, ..., n\\
\nonumber \  & 
  \begin{pmatrix}
0& -.5A_i^T\\
-.5A_i &  b_i
\end{pmatrix}
\bullet
\begin{pmatrix}
X & x\\
x^T &  x_0
\end{pmatrix}\geq 0, \  i= 1, ..., n\\
\nonumber \  & 
  \begin{pmatrix}
0 & 0\\
0 &  1
\end{pmatrix}
\bullet
\begin{pmatrix}
X & x\\
x^T &  x_0
\end{pmatrix}=1,\\
\nonumber \  & \ \ \ 
\begin{pmatrix}
X & x\\
x^T &  x_0
\end{pmatrix}\succeq 0.
\end{align}
  By $A_i^TA_j\bullet X=A_iXA_j^T$ and Shor decomposition, the  dual may be rewritten as 
 \begin{align}\label{DD0}
\nonumber \  \min \ & Q\bullet X +2c^Tx \\
 \    s.t.\   & \ AXA^T-b^TAx-x^TA^Tb\geq bb^T,\\
\nonumber \  &  Ax\leq b \\
\nonumber \  & \ \ \  X \succeq xx^T.
\end{align}
Problem \eqref{DD0} is a semi-definite relaxation of  \eqref{P}. Indeed, let $\bar{x}$ be a feasible point of  \eqref{P}, It is easily seen  that the matrix
$\begin{pmatrix}
\bar{x}\bar{x}^T & \bar{x}\\
\bar{x}^T &  1
\end{pmatrix}$
is feasible for \eqref{DD0} and $\bar x^TQ\bar x+2c^T\bar x=Q\bullet \bar x^T \bar x +2c^T \bar x$. \\
 \indent 
 It is worth noting that as problem \eqref{L} is feasible for any quadratic function, so the semi-definite program \eqref{L} is strongly feasible. By Theorem 3.2.6 in \cite{Ren}, the strong duality holds, that is, the optimal values of both problems \eqref{L} and \eqref{DD0} are equal. \\
 \indent
 As  problem \eqref{DD0} is the dual of problem  \eqref{L}, it follows from Proposition  \ref{w} that it is bounded from below, and there is no need to add more constraints. In general,  semidefinite relaxations are not necessarily  bounded from below \cite{Gorge}.  This relaxation is also called a strong relaxation by some scholars \cite{Bur3}. One may wonder to know about the relationship between this relaxation with other semidefinite relaxations existing in the literature. It is seen that problem \eqref{DD} is Shor's relaxation of 
\begin{equation*}
\begin{array}{ll}
 \  \min & \ x^TQx+2c^Tx
\\   s.t. &\ -Ax\geq -b,\\
&   \  (A_ix-b_i)(A_jx-b_j)\geq0, \ i\leq j=1, ..., m
\end{array}
\end{equation*}
which is exactly \eqref{P} with $0.5m(m+1)$ redundant constraints \cite{Shor}. Most relaxation methods, including RLT, use redundant constraints \cite{Sherali}. Gorge et al.  \cite{Gorge} use convex combination of these redundant constraints as a cut for semi-definite relaxations. We refer the reader to \cite{Gorge} for more information on the applications of these redundant constraints in QPs. Note that the constraint $Ax\leq b$ is redundant for \eqref{DD0} and so it can be removed  \cite{Sherali}. \\
\indent
 One important inquiry about bounds is  how one can reduce the gap. In this context, one idea may be the replacement of nonnegative affine functions with  nonnegative  convex quadratic functions on the given polytope. Similar to the affine case, a new bound can be formulated as a semi-definite program. Furthermore, one can show that most presented results in Section 2 hold in this case. However, the number of variables is of the order of  $m^3$, and makes this formulation less attractive. In addition, numerical implementations showed that the gap improvement was not considerable compared to the affine case. Pursuing this method by polynomials with degree greater than or equal to three is not practical since checking positivity of such a polynomial on a given polytope is not easy \cite{Las}. However, the procedure can be followed by restricting to  some classes of  polynomials \cite{Jiang}.\\
\indent
As mentioned earlier, problem \eqref{L} can be written as 
$$ \max_{q\in \mathcal{N}(X)}\{\min x^TQx+2c^Tx: q(x)\leq 0\},$$
 and  the bound \eqref{L} is exact provided  $x^TQx+2c^Tx-q^\star\in P[x]-\mathcal{N}(X)$ ($q^\star$ denotes the optimal value of \eqref{P}). As a result, enlargement of $\mathcal{N}(X)$ may lead to the reduction of  gap.\\
\indent
Let $\bar x\in int(X)$ and let $d\in \mathbf{R}^n$ be an arbitrary non-zero vector. The polytope $X$ can be partitioned in two polytopes $X_1=\{ d^T(x-\bar x)\leq 0, x\in X\}$ and $X_2=\{ d^T(x-\bar x)\geq 0, x\in X\}$. It is readily seen that
\begin{align*}
\mathcal{N}(X)\subseteq N(X_1)\cap N(X_2).
\end{align*}
It is likely that cone $\mathcal{N}(X)$  is strictly included in $\mathcal{N}(X_1)\cap \mathcal{N}(X_2)$  when the bound is not exact. To check $q\in \mathcal{N}(X_1)\cap \mathcal{N}(X_2)$, one needs to solve the linear system 
\begin{equation*}
\begin{array}{ll}
 \   & q(x)=\sum_{i=1}^{\hat m} \hat \alpha_i(x)(\hat A_ix-\hat b_i)\\
   \  & q(x)=\sum_{i=1}^{\bar m} \bar \alpha_i(x)(\bar A_ix-\bar b_i)\\
\  &  \hat \alpha_i\in \mathcal{A}_+(X_1), \ i=1, ..., \hat m \\
\  & \bar \alpha_i\in \mathcal{A}_+(X_2), \ i=1, ..., \bar m 
\end{array}
\end{equation*}
where $X_1=\{x\in \mathbf{R}^n: \hat A_ix-\hat b_i, \ 1\leq i \leq \hat m\}$ and $X_2=\{x\in \mathbf{R}^n: \bar A_ix-\bar b_i, \ 1\leq i \leq \bar m\}$. 
Thus, by the replacement of
$\mathcal{N}(X)$ with $\mathcal{N}(X_1)\cap \mathcal{N}(X_2)$ in \eqref{L}, the number of variables and constraints will be two times more than the former case. It is easily seen the optimal value of \eqref{L} depends on the choice of $\bar x$ and $d$. Note that partitioning the feasible set is a wide-spread method for reducing the duality gap; See \cite{Jam, Spon} and references therein.  We will take advantage of this idea to develop a branch and cut algorithm for concave QPs.\\
\indent
 In the same line, one can partition $X$ to $k$ polytopes and fattens $\mathcal{N}(X)$. In this case, the number of variables  will be of $O(km^2)$. For instance, for $\bar x\in int(X)$, one could consider two different hyperplanes which pass through the given point. As a result, the polytope is divided into four polytopes.\\
\indent
Although problem \eqref{L} provides a lower bound for (QP), the number of variables is of $O(m^2)$, which makes this semi-definite program time-consuming in some cases. In the sequel, we propose the bounds whose decision variables  are less than that of problem \eqref{L}. However, this problem does not equip us with a better lower bound.\\
\indent
Consider problem \eqref{P}. Let $L$ denote  the subspace generated by the eigenvectors of $Q$ corresponding to the negative eigenvalues.  We propose the following bound for  \eqref{P},
\begin{equation}\label{L1}
\begin{array}{ll}
 \ & \max \ \ell \\
 \  &  s.t.\   x^TQx+2c^Tx-\ell+\sum_{i=1}^{m} \alpha_i(x)(A_ix-b_i)\in P[x],\\
\  & \ \ \alpha_i\in \mathcal{A}_+(X), \ i=1, ..., m,\\
\  & \ \ \nabla\alpha_i\in L, \ i=1, ..., m.
\end{array}
\end{equation}
\indent
The above problem is reduced to problem \eqref{L} if $Q$ is negative definite. Nevertheless, for the class of problems which $Q$ has only one negative eigenvalue, problem \eqref{L1} has $2m$ variables. As a result, it would be more beneficial from time aspect to tackle this problem instead of \eqref{L}. In the following proposition, we prove that the optimal value of problem \eqref{L1} is finite.
\begin{proposition}
Let $X$ be a polytope. Then problem \eqref{L1} has finite optimal value.
\begin{proof}
By Proposition \ref{w},  the optimal value of  \eqref{L1} is either finite or minus infinity. So, it suffices to prove the existence of a feasible point. Let $\{v_1, ..., v_k\}$ be a basis for $L$. As $X$ is bounded, by virtue of  Farkas' Lemma, there exist nonnegative constants $Y_{ij}$, $j=1, ..., k$, such that
$v_j=\sum_{i=1}^{m}Y_{ij} A_i^T$. Moreover, there are $f_j$, $j=1, ..., k$, with $v_j^Tx+f_j\in \mathcal{A}_+(X)$. For $\gamma$ sufficiently large, the matrix $Q+\gamma\sum_{j=1}^{k}v_jv_j^T$ is positive semi-definite. As $X$ is bounded, for suitable choice of $\ell$ the affine functions
$\alpha_i(x)=\gamma\sum_{j=1}^{k}Y_{ij}( v_j^Tx+f_j)$, $i=1, ...,m$, fulfill all constraints of \eqref{L1}.
\end{proof}
\end{proposition}

\indent
The following example demonstrates that the optimal value of  \eqref{L1} may be strictly less than that of  \eqref{L}.
\begin{example}
Consider the QP,
\begin{equation*}
\begin{array}{ll}
 \ & \min \ 2x_1x_2 \\
 \  &  s.t.\   x_1, x_2\leq 1\\
\  & \  \ \ -x_1, -x_2\leq 0.
\end{array}
\end{equation*}
The optimal values of this QP and  \eqref{L} is zero (Take into account $\alpha_1(x)=0, \alpha_2(x)=0, \alpha_3(x)=x_2, \alpha_4(x)=x_1, \ell=0$).  It is seen that $\alpha\in \mathcal{A}_+(X)$ and $\nabla\alpha\in L$ if and only if $\alpha\in cone(\{x_1-x_2+1, -x_1+x_2+1\})$. By solving problem \eqref{L1}, for any feasible point, $\alpha_1, \alpha_2, \alpha_3, \alpha_4, \ell$, we have $\ell\leq \frac{-1}{8}$.
\end{example}

\indent
One can  also formulate a semi-definite program for  box constrained QPs  with fewer variables in comparison with \eqref{L}. In this case, one could consider the affine coefficient of constraint $x_k\leq u_k$ ($-x_k\leq -l_k$), $\alpha_k$, in the form
$$
\alpha_k(x)=d_kx_k+f_i, \   \alpha_k\in \mathcal{A}_+(X),
$$
where $d_k\in \mathbf{R}$. Similar to Proposition \ref{well}, it is proved that the bound  is  finite in this case as well.\\
\section{Comparison with existing bounds}
One important question here is the relationship between  the  bound  \eqref{L} and the conventional lower bounds for QPs. For comparison, we focus on two types of QPs,  standard quadratic programs and box constrained quadratic programs. Let us first concentrate on standard quadratic programs.  Consider the standard quadratic program,
\begin{equation}\tag{StQP}\label{SQ}
\begin{array}{ll} 
 \ & \min \ x^TQx
\\ &  s.t. \ \sum _{i=1}^n x_i= 1,\\
&  \ \  x\geq 0.
\end{array}
\end{equation}

\indent
It is well-known that \eqref{SQ} is solvable in polynomial time if $Q$ is either positive semi-definite or negative semi-definite on standard simplex (denoted by $\Delta$ hereafter). In general, \eqref{SQ} is NP-hard \cite{Bomze1}. \\
\indent
 We denote the optimal value of \eqref{SQ} by $\ell_Q$.  
Note that optimizing a quadratic function on standard simplex can be casted as \eqref{SQ}. This follows form the fact that for each $x\in \Delta$, we have 
$x^TQx+2c^Tx=x^T(Q+ec^T+ce^T)x$.
It is seen that  $\ell_{Q+tee^T}=\ell_Q+t$ for each $t\in\mathbf{R}$. So, in the current section  we assume that $Q$ is nonnegative in \eqref{SQ}.\\
\indent
Bomze et al. \cite{Bomze1} have proposed the best (quadratic) convex underestimation bound as follows 
\begin{equation}\label{conv}
\ell_Q^{conv} =\sup\{\ell_S :S\succeq 0, Q-S\geq 0, diag(S)=diag(Q)\},
\end{equation} 
where $diag(S)$ denotes the diagonal of $S$.  They show that the above problem gives better bound in comparison with other quadratic bounds.  Problem \eqref{L} is formulated for \eqref{SQ} as follows,
\begin{equation}\label{DD}
\begin{array}{ll} 
 \ & \max \ \ell \\
 \  &  s.t.\   x^TQx-\ell+\sum_{i=1}^{n} \alpha_i(x)(-x_i)+\alpha_{n+1}(x)(e^Tx-1)\in P[x],\\
\  & \ \  \ \alpha_i\in \mathcal{A}_+(\Delta),\ i=1, ..., n.
\end{array}
\end{equation}
\indent
 Next theorem shows that bounds \eqref{conv} and \eqref{DD} are equivalent. 
\begin{theorem} 
Problems \eqref{conv} and \eqref{DD} give the same bound. 
\begin{proof} 
Let $\bar\ell$ denote the optimal value problem \eqref{DD}. First we show that $\ell_Q^{conv}\leq \bar \ell$. Without loss of generality, we may assume that  \eqref{conv} admits an  optimal solution $S$. As $(Q-S)_i\geq 0$, 
$\bar \alpha_i(x)=(Q-S)_ix\in \mathcal{A}_+(\Delta)$, $i=1, ..., n$, and 
$$
x^TQx-x^TSx+\sum_{i=1}^{n} \bar\alpha_i(x)(-x_i)=0.
$$
Invoking the optimality conditions for convex QPs, there are nonnegative scalars $\beta_i$, $i=1, ..., n$ and $\beta_{n+1}$ such that
$$
 x^TSx-\ell_Q^{conv}+\sum_{i=1}^{n} \beta_i(-x_i)+\beta_{n+1}(e^Tx-1)\in P[x]. 
$$
By above equalities, it is seen that $\alpha_i(x)=\bar \alpha_i(x)+\beta_i$, $i=1, ..., n+1$, and $\ell=\ell_Q^{conv}$ are feasible for problem \eqref{DD}. So $\ell_Q^{conv}\leq \bar \ell$.\\
 Now, let $\bar\ell$ and $\bar \alpha_i=a_i^Tx+a^0_i$, $i=1, ..., n+1$, be optimal  for \eqref{DD}. We get
\begin{equation}\label{lkj}
x^TQx-\bar \ell -\sum_{i=1}^n x_i \alpha_i(x)+(e^Tx-1)\alpha_{n+1}(x)=(x-\bar x)^TS(x-\bar x)+s,
\end{equation}
where $S$ and $s$ are a positive semi-definite matrix and a nonnegative scalar, respectively, and $\bar x\in\mathbf{R}^n$. Let $e^Tx\neq 0$. By replacing $x$ with $(e^Tx)^{-1}x$  and multiplying  both sides of \eqref{lkj} by 
$(e^Tx)^2$, we get 
$$
x^T(Q-\bar\ell ee^T)x-\sum_{i=1}^n x^T e_i(a_i+a^0_i e)x=(x-(e^Tx)\bar x)^TS(x-(e^Tx)\bar x)+s((e^Tx))^2.
$$
Since $\alpha_i\in\mathcal{A}_+(\Delta)$, $a_i^j+a^0_i\geq 0$ for $i=1, ..., n$ and  $j=1, ..., n$. So, matrix $N=\sum_{i=1}^n  e_i(a_i+a^0_i e)\geq 0$. By continuity,  the homogenous quadratic function $(x-(e^Tx)\bar x)^TS(x-(e^Tx)\bar x)+s((e^Tx))^2$ is nonnegative on $\mathbf{R}^n$. Thus,  $(x-(e^Tx)\bar x)^TS(x-(e^Tx)\bar x)+s((e^Tx))^2=x^T\bar S x$ for some $\bar S \succeq 0$. Hence,
 $$
 \{\max \ell: Q-\ell ee^T \succeq N, \ N\geq 0\}\geq \bar\ell.
 $$
  Since  $\ell_Q^{conv}=\{\max \ell: Q-\ell ee^T \succeq N, \ N\geq 0\}$ (see Section 6 in \cite{Bomze1}), $\bar \ell \leq \ell_Q^{conv}$ and the proof is complete.
\end{proof} 
\end{theorem} 

\indent
In the rest of the section, we continue our discussion for box constrained QPs. Consider the box constrained QP 
\begin{equation}\label{CP}
\begin{array}{ll} 
 \ & \min \  x^TQ_0x+2c_0^Tx \\ 
  & \; \; \; \; \; \; \; \; \; a_i^Tx=d_i, \ \ \;\;\;\;\;\;\;\;\;\;\;\;\; i=1, ..., m \\
  & \; \; \; \; \; \; \; \; \;   l\leq x\leq u,
\end{array}
\end{equation}
where $n>m$. For convenience, we may assume that $l=0$ and $u=e$. By the combination of semidefinite programming relaxation (SDP) and RLT, some scholars have proposed new relaxations for problem \eqref{CP} \cite{Ans, Bao}. One of the most effective relaxations in this category is Shor relaxation with partial first-level RLT  (SRLT). This bound is formulated as 
 \begin{align}\label{D2}
\nonumber \ & \min \ Q_0\bullet X+2c_0^Tx \\
\nonumber  \   &  s.t.\    Xa_i^T= d_ix, \ \ \;\;\;\;\;\;\;\;\;\;\;\;\;i=1, ..., m \\
\nonumber  \  & \; \; \; \; \; \; \; \; \; a_i^Tx=d_i, \ \ \;\;\;\;\;\;\;\;\;\;\;\;\; i=1, ..., m \\
& \; \; \; \; \; \; \; \; \; \; \; \;    X\geq 0, \ \\
\nonumber  & \; \; \; \; \; \; \; \; \; \; \; \;    ex^T-X\geq 0,\\
\nonumber  & \; \; \; \; \; \; \; \; \; \; \; \;    X-ex^T-xe^T+ee^T\geq 0,\\
\nonumber & \; \; \; \; \; \; \; \; \; \; \; \;   X-xx^T\succeq 0.
\end{align}
Bao et al.  \cite{Bao} have provided a full comparison of relaxation methods for box constrained QPs. They show that  SRLT dominates the other relaxations. By the discussion made in the former section on the dual of \eqref{L}, it is straightforward to see that SRLT is the dual of \eqref{L} for box constrained QPs, see \eqref{DD0}. As strong duality holds for problem \eqref{L},  SRLT snd \eqref{L} are equivalent.\\
\indent
One important issue with quadratic programs is how to convert a relaxation solution to an approximate  solution \cite{Luo}. As problem \eqref{L} not only provides a lower bound for quadratic programs, but also gives a convex underestimator, one may obtain an approximate solution by optimizing the given function on the feasible set. We use this strategy in the next section.\\
\indent
We conclude the section by addressing an interesting point about bound \eqref{L}. This bound can be regarded as a special case of the following bound
 \begin{align}\label{D20}
\nonumber \ & \max \  \ell \\
 \   &  s.t.\    x^TQx+2c^Tx-\ell -\sum_{\tau\in \mathbb{N}^d_d}\lambda_\tau\prod_{i=1}^d (b_i-A_ix)^{\tau_i}=\sigma(x),\\
\nonumber  \  & \; \; \; \; \; \; \; \; \; \sigma\in\Sigma[x], \ \lambda\geq 0,
\end{align}
where $\mathbb{N}_d^d=\{\tau \in \mathbb{N}^d: \sum_{i=1}^d \tau_i\leq d \}$. One may regard the above-mentioned bound as a combination of  Lasserre hierarchy and Handelman's approximation hierarchy. One can obtain bound \eqref{L} by setting $d=2$ in \eqref{D20}. In fact, this follows from Non-homogenous  Farkas' Lemma. Recently, some scholars have taken advantage of this idea and proposed new bounds for general polynomial optimization problems; See \cite{Lass} for more details.  
\section{A new algorithm for concave quadratic optimization}
In this section, by virtue of the newly introduced bound, we propose a new algorithm for concave QPs. Throughout the section, it is assumed that $X$ is a polytope with nonempty interior and $Q$ is negative semi-definite. We introduce a branch and cut (B\&C)  algorithm. We use Konno's cut in the cutting step. Before we go into the details of the algorithm, let us remind a definition.
\begin{definition}
Let $\hat x\in X$ be a vertex. This vertex is called a local optimal if the value of the objective function at this point is less than or equal to that at  adjacent vertices.
\end{definition}

\indent
 As mentioned before, we are developing a B\&C algorithm for concave QPs. The main steps of a B\&C method are branching, bounding, fathoming and cutting. A typical branching approach for QPs is partitioning and for bounding is linear program relaxation based on RLT; See \cite{Horst, Sherali2} for more details. Recently, some scholars have employed KKT optimality condition and semi-definite relaxation for branching and bounding, respectively \cite{Bur2, Chen}.  In the sequel, we present the details of the steps.\\
 \indent
The proposed method regards \eqref{P} as the root node of the B\&C tree. Let the following quadratic program be the subproblem of a node,
  \begin{equation*}
\begin{array}{ll}
 \ & \min \ x^TQx+2c^Tx
\\ &  s.t. \ \bar Ax\leq \bar b,
\end{array}
 \end{equation*}
 and let $\bar{X}$ denote the feasible set of the above problem.
 To get a lower bound for the node, we formulate  the semi-definite program
 \begin{align} \label{LLu}
\nonumber \max \  \ & \ell \\
 \  \nonumber  s.t.\  \ &
 \begin{pmatrix}
-\bar A^T Y \bar A &  \  \bar A^T(Y \bar b+0.5y)\\
(Y \bar b+0.5y)^T\bar A &   \ -y^T\bar b-\bar b ^TY \bar b
\end{pmatrix}
+\begin{pmatrix}
Q & c\\
c^T &  -\ell
\end{pmatrix}
\succeq 0,\\
 & \ \ell\leq u\\
\nonumber \ & Y \geq 0, \ Y=Y^T, \ y\geq 0,
\end{align}
 where $u$ is the best upper bound obtained by the algorithm so far. Let $\bar{Y}$ and  $\bar{y}$ be optimal solutions to \eqref{LLu}. Then, we formulate the convex QP,
\begin{equation}\label{Qc}
\begin{array}{ll}
 \ & \min \ x^T(Q-\bar A^T \bar Y \bar A)x+2(c+\bar A^T\bar Y \bar b+0.5A^T\bar y)^Tx
\\ &  s.t. \ Ax\leq b.
\end{array}
\end{equation}

\indent
Let $\hat x$ be a solution of problem \eqref{Qc}. Next, a local vertex optimal point $\bar{x}\in X$ is chosen such that $\bar{x}^TQ\bar{x}+2c^T\bar{x}\leq \hat{x}^TQ\hat{x}+2c^T\hat{x}$. This step is not time consuming. Indeed, there are some efficient approaches for computing $\bar{x}$ \cite{Tuy}. We use the value $\bar{x}^TQ\bar{x}+2c^T\bar{x}$ to update the upper bound, $u$. We use $\bar{l}$ and other lower bounds obtained from fathomed nodes and child nodes to update the lower bound.\\
\indent
After the bounding step, if the difference of $\ell$ and $u$ is less than the prescribed tolerance, $\epsilon$, the node will be fathomed. Otherwise, the algorithm solves problem \eqref{Qc} corresponding to the feasible set of the subproblem, $\bar X$, and computes a local optimal point $\bar x \in \bar X$ for concave QP $\{\min x^TQx+2c^Tx: x\in \bar X\}$.\\
\indent
In the next step, the algorithm produces a cut. Let us go into the details of the step. As mentioned above, we employ Konno's cut. For convenience, let $\bar x=0$ be a non-degenerated  local optimal vertex of $\{\min x^TQx+2c^Tx: x\in \bar X\}$. Suppose that  vectors $e_i$,  $i=1, ..., n$, denote extreme directions at $\bar x\in \bar X$.  Let $q(x)=x^TQx+2c^Tx$. To compute Konno's cut, first we need to obtain Tuy's cut given by
$$
\sum_{i=1}^n \frac{x_i}{t_i}\geq 1,
$$
 where $t_i:=\max\{\theta: q(\bar x+\theta e_i)\geq u-\epsilon\}$. Then, the algorithm computes $y^i\in \argmin\{\bar x^TQy+2c^Ty: x\in X, \sum_{i=1}^n \frac{x_i}{t_i}\geq 1\}$. If  the following inequality holds it will continue the cutting step, for otherwise it goes to the branching step.
 $$
 \min_{i=1, ..., n} \{q(y^i)\}\geq u-\epsilon.
 $$
We call $\bar x$ an eligible vertex if we have the above inequality. Let $\bar x$ be eligible. The method will add the valid cut
 $
\sum_{i=1}^n \frac{x_i}{s_i}\geq 1,
$
where
$$
s_i=\max\{\theta: -A^T+t^{-1}\mu-Q_i^T\theta=c, -b^T\lambda+\mu+c_i\theta\geq u-\epsilon, \lambda\geq 0, \mu\geq 0\},
$$
for $i=1, ..., n$ and $t^{-1}=\begin{pmatrix}
 \frac{1}{t_1}, &
\hdots &,
 \frac{1}{t_n}
\end{pmatrix}^T$. We refer the interested reader to \cite{Konno, Tuy} for more information on Konno's cut.\\
\indent
After computing Konno's cut, the method updates the feasible set, that is, set $\bar X=\{x\in \bar X: \sum_{i=1}^n \frac{x_i}{s_i}\geq 1\}$. If the feasible set is empty, the node will be fathomed. Otherwise, the method goes to the branching step.\\
\indent
  For branching, the algorithm divides the polytope $\bar X$ into two partitions $X_1$ and $X_2$ such that $X_1=\{x: x\in \bar X, d^Tx\leq d^T x_c\}$ and $ X_2=\{x: x\in \bar X, d^Tx\geq d^T x_c\}$ where $x_c$ is a Chebyshev center of $\bar X$ and $d\neq 0$ is a random vector in $\mathbf{R}^n$. It is worth noting that  a Chebyshev center of a polytope is computed by solving a linear program\cite{Boyd}. \\
Updating the upper bound is straightforward (it is enough to consider the minimum of the provided upper bounds).
To update the lower bound, we must consider the lower bound of all fathomed and child nodes. Strictly speaking, it is seen that for father node $n_p$ and its two child nodes, $n_{p_1}$ and $n_{p_2}$, we have  $l_p\leq \min\{l_{p_1}, l_{p_2}\}$, where  $l_p, l_{p_1}$ and $l_{p_2}$ denote the generated lower bound of $n_p, n_{p_1}$ and $n_{p_2}$, respectively. Therefore, if $\{ l^f_i\}_{i=1}^k$ and $\{ l^c_i\}_{i=1}^o$ are the lower bound fathomed and new nodes, respectively,  then the new lower bound will be obtained by the following formula
\begin{equation}\label{upd}
l=\min\{\min_{i=1}^k  l^f_i, \min_{i=1}^o  l^c_i \}.
\end{equation}
It can be seen that the lower bound is increasing throughout the algorithm. All steps of the method are presented in Algorithm 1.
\begin{algorithm}[]\label{Alg}
  \caption{Branch and Cut Algorithm}
  \begin{algorithmic}[1]
    \Initialize{$k=1, l=-\infty, u=\infty, s=1, \epsilon >0, L=\{\}, \Pi_1=\{(P_1, X)\}$ and $\Pi_2=\{\}$. }
    \While{s = 1}
        \While{$\Pi_1\neq\emptyset$}
      \State $(P_k, X_k) \gets select(\Pi_1)$ and $\Pi_1\setminus (P_k, X_k)$.
     \State Solve semi-definite program \eqref{L} corresponding to $(P_k, X_k)$ to get $l_k$.
        \State $L \gets l_k$ and delete its father lower bound if exists in $L$.
        \State Solve convex QP \eqref{Qc} and obtain a local vertex optimal $\bar x$.
        \State $u\gets \min(u, \bar x^TQ \bar x+2c^T\bar x$).
        \If{$l_k+\epsilon< u,$} 
                 \State \begin{varwidth}[t]{\linewidth} 
                    Solve convex QP \eqref{Qc} corresponding to $X_k$ and obtain a local vertex optimal \hspace{5cm}
                    $\bar x\in\{\min x^TQx+2c^Tx: x\in \bar X\}$.
                    \end{varwidth}
                 \If{$\bar{x}$ is eligible,} 
                    \State Compute Konno's cut $\sum_{i=1}^n \frac{x_i}{t_i}\geq 1$ at $\bar x$.
                    \State  $ X_k\gets\{ X_k: \sum_{i=1}^n \frac{x_i}{t_i}\geq 1\}$
                 \EndIf
                 \If{$X_k\neq \emptyset$,}
                    \State Branch  $(P_k, X_k)$ to two nodes  $(P_{k+1}, X_{k+1})$ and $(P_{k+2}, X_{k+2})$.
                     \State $\Pi_2\gets\{(P_{k+1}, X_{k+1}), (P_{k+2}, X_{k+2})\}$ and $k\gets k+2$.
                 \EndIf
            \EndIf
       \EndWhile
        \State update $l$ by formula \eqref{upd}.
         \If{$\Pi_2=\emptyset$ or stopping criteria are satisfied,}
        \State $s=0$.
        \Else
        \State $\Pi_1\gets \Pi_2$ and $\Pi_2\gets \{\}$.
      \EndIf
    \EndWhile
  \end{algorithmic}
\end{algorithm}

\indent
Here, $\Pi_1$ and  $\Pi_2$ denote the set of nodes and the set $L$ contains lower bounds of fathomed and child nodes. Additionally,  $P_k$ and  $X_k$ denote the $k^{th}$ node and its feasible set, respectively. Stoping criteria which one may use in Algorithm 1 can be  absolute gap tolerance, a limit on the maximum running time, etc.\\
\indent
In the rest of the section, we investigate the finite convergent of the algorithm with the stopping accuracy
$\epsilon> 0$. To this end, we consider function $\Phi: \mathbf{R}^n\times \mathbf{R}_{+}^n\to \mathbf{R}$ where $\Phi(y, d)$ is defined as the optimal value 
\begin{align}\label{ps}
\nonumber & \max \  \ell \\
\nonumber \  &  s.t.\   x^TQx+2c^Tx -\ell+\sum_{i=1}^{n} [\alpha_i(x)(x-y_i-d_i)+\mu_i(x)(-x+y_i-d_i)]\in P[x]\\
 \ &  \ \ \  \  \  \alpha_i, \mu_i \in \mathcal{A}_+(X(y,d)),
\end{align}
where $X(y,d)=\{x: y-d\leq x\leq y+d\}$. The well-definedness of $\Phi$ on its domain follows from Proposition \ref{well}. The next lemma lists some properties  of $\Phi$.
\begin{lemma} \label{lcon}
The function $\Phi: \mathbf{R}^n\times \mathbf{R}_+^n\to \mathbf{R}$ has the following properties.
\begin{itemize}
\item[(i)]  $\Phi(y, 0)=y^TQy+2c^Ty ,  \ \ \forall y\in \mathbf{R}^n$;
\item[(ii)]  $\Phi$ is continuous on $X\times  \mathbf{R}_{+}^n$.
\end{itemize}
\begin{proof}
The first part  follows immediately from Proposition \ref{sing}. \\
For the second part, first we prove the lower semi-continuity of $\Phi$.  Let $(\bar y, \bar d)\in   \mathbf{R}_{+}^n$ and $\bar \alpha_i, \bar\mu_i \in \mathcal{A}_+(X(\bar y, \bar d))$ and $\bar\ell$ are optimal for \eqref{ps}. For every $\epsilon>0$, there are $\hat \alpha_i, \hat\mu_i \in \mathcal{A}_+(X(\bar y, \bar d))$ such that the quadratic function
$$
x^T Qx+2 c^Tx -(\bar \ell-\epsilon)+\sum_{i=1}^{n} (\bar\alpha_i(x)+\hat\alpha_i(x))(x-\bar y_i-\bar d_i)+(\bar\mu_i(x)+\hat\mu_i(x))(-x+\bar y_i- \bar d_i),
$$
is strictly convex and positive on $\mathbf{R}^n$. As a result, for small perturbations of $\bar \alpha_i, \bar\mu_i, \bar y$ and $\bar d$, the above-mentioned quadratic function belongs to  $P[x]$, which implies
$$
\liminf_{(y,d)\to (\bar y, \bar d)}\Phi(y,d)\geq \bar \ell-\epsilon.
$$
As the above atatement holds for each $\epsilon>0$, $\Phi$ is lower semi-continuous  at $(\bar y, \bar d)$.
Now, we prove the upper semi-continuity of $\Phi$. First, we consider the case $\bar d \in int(\mathbf{R}_{+}^n)$. Let the sequence
$\{(y_k, d_k)\}\subseteq X\times   \mathbf{R}_{+}^n$ tends to $(\bar y, \bar d)$. Suppose that $\alpha_i^k, \mu_i^k$ and $\ell^k$ are optimal for problem \eqref{ps} corresponding to $(y_k, d_k)$. If for each $i=1, ..., n$ the sequences $\{\alpha_i^k\}, \{\mu_i^k\}\subseteq \mathbf{R}^{n+1}$ are bounded, then without loss of generality we may assume that $\alpha_i^k\to \bar \alpha_i$, $\mu_i^k\to \bar \mu_i$ and $\ell_k\to \bar \ell$.  In addition, due to the lower semi-continuity of the set-valued mapping $ X(., .)$, we have $\bar \alpha_i, \bar \mu_i \in \mathcal{A}_+(X(\bar y,\bar d))$ and
$$
x^T Qx+2 c^Tx -\bar \ell+\sum_{i=1}^{n} [\bar\alpha_i(x)(x-\bar y_i-\bar d_i)+\bar\mu_i(x)(-x+\bar y_i- \bar d_i)]\in P[x],
$$
which implies upper semi-continuity in this case. For the case of the existence of some unbounded sequences, without loss of generality we may assume that  $t_k^{-1}\alpha_i^k\to \bar \alpha_i$ and $t_k^{-1}\mu_i^k\to \bar \mu_i$, where $t_k=\max_{1\leq i\leq m}\{\|\alpha_i^k\|, \|\mu_i^k\|\}$. Moreover, there exists $\bar\alpha_i\neq 0$ and
$$
q(x)=\sum_{i=1}^{n} [\bar\alpha_i(x)(x-\bar y_i-\bar d_i)+\bar\mu_i(x)(-x+\bar y_i- \bar d_i)]\in P[x].
$$
Similar to the proof of Proposition \ref{popt}, because of $int(X(\bar y,\bar d))\neq \emptyset$, there is $\bar x$ such that $q(\bar x)<0$, which contradicts the nonnegativity of $q$. So, the unboundedness case cannot occur, and in this case the upper semi-continuity of $\Phi$ is derived. Likewise, one can prove it for the case that some components of $\bar d$ are zero, $\Phi$ is upper semi-continuous at $(\bar y, \bar d)$ on $X\times  \mathbf{R}_{\bar d}^n$, where $\mathbf{R}_{\bar d}^n=\{x\in \mathbf{R}_{+}^n: x_i=0 \ if \ \bar d_i=0\}$. Suppose that the sequence
$\{(y_k, d_k)\}\subseteq X\times   \mathbf{R}_{+}^n$ tends to $(\bar y, \bar d)$. We decompose the sequence $\{d_k\}$ as $d_k=d^1_k+d_k^2$ where $d_k^1$ is the projection of $d_k$ on $\mathbf{R}_{\bar d}^n$. It is  seen that $d^1_k\to \bar d$. So, we have
$$
\limsup_{k\to \infty}\Phi(y_k,d_k)\leq \limsup_{k\to \infty}\Phi(y_k,d_k^1)\leq \Phi(\bar y,\bar d).
$$
The first inequality results from the inclusion property. Therefore, $\Phi$ is continuous on  $X\times  \mathbf{R}_{+}^n$ and the proof is complete.
\end{proof}
\end{lemma}
\begin{theorem}
Algorithm 1  is finitely convergent with the stopping accuracy \ \  $\epsilon>0$.
\begin{proof}
Consider the function $\Phi$ on the compact set $X\times (B\cap \mathbf{R}_{+}^n)$, where $B$ stands for the closed unit ball. On account of  Lemma \ref{lcon}, we can infer uniform continuity of $\Phi$ on  the given domain. Additionally, we have the following property
$$
\forall \epsilon>0, \ \exists \delta>0, \forall x\in X, \forall d\in  \mathbf{R}_{+}^n; \ \|d\|_{\infty}<2\delta \ \Rightarrow  |\Phi(x, d)-x^TQx-2c^Tx|<\epsilon.
$$
Since $X$ is compact, the objective function is Lipschitz continuous on it. Without loss of generality,  let Lipschitz modulus be one. Let $\Delta\subseteq X$ be a polytope with diameter less than
$0.5\min\{\epsilon, \delta\}$, that is, $\max_{x_1, x_2\in \Delta} \|x_1-x_2\| <0.5\min\{\epsilon, \delta\}$. As a result, there are $\bar y\in X$ and $\bar d\in  \mathbf{R}_{+}^n$ such that $\Delta\subseteq X(\bar y, \bar d)$ and $\|\bar d\|_{\infty}<2\delta$. Due to the inclusion property and the provided results, we have
\begin{align*}
& | opt(\Delta, Q, c)-\min_{x\in\Delta}(x^TQx+2c^Tx) | \leq |\Phi(\bar y, \bar d)-\min_{x\in\Delta}(x^TQx+2c^Tx) | \\
 & \ \leq |\Phi(\bar y, \bar d)-\bar y^TQ\bar y-2c^T\bar y|+|\bar y^TQ\bar y+2c^T\bar y-\min_{x\in\Delta}(x^TQx+2c^Tx) | \\
 &  \ \leq \epsilon+\epsilon=2\epsilon.
\end{align*}
Since after finite number of branching, the feasible set of subproblems, $\Delta$, is included in $X(y, d)$ for some $y\in X$ and $d\in  \mathbf{R}_{+}^n$ ($\|d\|_{\infty}<\min\{\epsilon, \delta\}$), all nodes will be fathomed and algorithm will stop after finite steps.
\end{proof}
\end{theorem}

\indent
It is worth noting that for having the finite convergent, one should adopt a branching procedure which guarantees the diameters of generated polytopes tend to zero. However, if the method selects nonzero vector $d$ randomly, with probability of one, the diameters of generated polytopes will converge to zero.
\section{Computational results}
In this section, we illustrate numerical performance of Algorithm 1 on four groups of test problems.  The code and the test problems are publicly available at {https://github.com/molsemzamani/quadproga}.\\
\indent
We implemented the algorithm using MATLAB 2018b. The computations were run on a Windows PC with  Intel Core i7 CPU, 3.4 GHz, and 16GB of RAM. To solve semi-definite program \eqref{L}, we employed  MOSEK \cite{MOSEK}. To solve convex QP \eqref{Qc}, we employed CPLEX's function cplexqp. In addition, CPLEX was used  for solving linear programs.\\
\indent
To evaluate the performance of Algorithm 1, we compared the numerical results with three non-convex quadratic optimization solvers: BARON 18.11.12, Couenne  v. 1.0 and CPLEX 12.8 \cite{BARON, Couenne, Cplex}. All solvers were run on MATLAB 2018b and we applied AMPL to pass the problems to BARON and Couenne \cite{AMPL}. \\
\indent
In our numerical experiments, we used two stopping criteria, absolute gap tolerance and running time limit, to terminate the solvers. The  absolute gap is defined as a difference between the given lower and upper bounds. \\
\indent
For the first group, we selected twenty concave instances from Globallib folder in \cite{Chen}. This folder contains all non-convex instances of  Globallib test problems \cite{Globallib}. The dimension of problems range from five to fifty. \\
\indent
We set the absolute gap tolerance and the maximum running time  to $10^{-4}$ and $100$ seconds, respectively. Since all methods could give us global optimum with the prescribed gap, we just  report the running time. The performance of all solvers are summarized in Table \ref{Globallib}, which  $n$ denotes the dimension of instances and the rest columns denote the execution time for solvers. \\
\begin{table*}[h]
\caption{Globallib instances}
\label{Globallib}
\begin{tabular}{c c c c c c c c c c}
   \toprule
  Instance  &  n & BARON  & Couenne & CPLEX &Algorithm 1\\
  \hline
     \multicolumn{1}{l}{st-qpc-m1} & 5 & 0.08 & 0.12 & 0.04 & 0.11  \\
         \multicolumn{1}{l}{st-bsj4}& 6	&0.23&	0.09&	0.05&	0.31\\
        \multicolumn{1}{l}{ex2-1-6} &10	&0.29&	0.19&	0.05&	0.56\\
     \multicolumn{1}{l}{st-fp5}  &10	&0.14&	0.11&	0.07&	0.15	 \\
       \multicolumn{1}{l}{st-qpk3}  & 11	 &0.39	 &1.19	 &0.07	 & 0.14	  \\
           \multicolumn{1}{l}{qudlin}  &12	&0.13&	0.1&	0.01&	0.12\\
    \multicolumn{1}{l}{ex2-1-7} &20	&0.76	&10.06	&0.09&	1.38\\
    \multicolumn{1}{l}{st-fp7a} &20	&0.55	&0.58&	0.08&	0.52\\
    \multicolumn{1}{l}{st-fp7b}  &20	&0.45&	0.89&	0.07&	0.48 \\
    \multicolumn{1}{l}{st-fp7d}  & 20	 &0.31	 &0.46	 &0.06	 &0.18	  \\
      \multicolumn{1}{l}{st-fp7e}  &20	&0.81	&10.13	&0.13&	 1.81 \\
    \multicolumn{1}{l}{st-m1}  & 20	& 0.26	& 0.18	& 0.08	& 0.24	\\
     \multicolumn{1}{l}{ex2-1-8}  &24	&0.19&	0.08	&0.01&	0.19	 \\
    \multicolumn{1}{l}{st-m2}  & 30	&0.41	&0.35&	0.18	&0.65	\\
        \multicolumn{1}{l}{st-rv7} & 30	&0.49&	0.87	&0.12	&0.41	 \\
       \multicolumn{1}{l}{st-rv8}  & 40	&0.58&	0.77&	0.13&	1.05\\
       \multicolumn{1}{l}{st-rv9}  & 50	&2.1	&3.69	&0.43	&1.8\\
    \midrule
    \bottomrule
\end{tabular}
\end{table*}
\indent
The second group of examples involves twenty concave QPs with dense data. The test problems were generated as follows. The feasible set, $X$, was given by the following linear system
$$
Ax\leq 10b,\\
\sum_{i=1}^nx_i\leq 100,
x\geq 0,
$$
where square matrix $A$ and and vector $b$ were generated by MATLAB's function randn and rand, respectively. The randn function generates a sample of a Gaussian random variable, with mean 0 and standard deviation 1, while rand generates a uniformly distributed random number between 0 and 1. We also generated the vector $c$ via randn function. We generated the square matrix $Q$ with the formula $Q=-U^TDU$, where $U$ is an orthogonal matrix obtained form the singular value decomposition of some random matrix and $D$ is a diagonal matrix  whose components are chosen by rand function.  We generated twenty concave QPs in $\mathbf{R}^{40}$ and $\mathbf{R}^{45}$. \\
\indent
We set the absolute gap tolerance and the maximum running time to $10^{-3}$ and $1000$ seconds, respectively.  We report the generated lower bound and running time.  If for some instance the running time is less than $1000$ seconds, the solver succeeded in solving with the prescribed gap. Table \ref{dense} reports computational performances of  both methods. In this table, $q^\star$ denotes the optimal value and columns $lb$ and $time$ show the lower bound and the spent CPU time, respectively. To evaluate the quality of the generated upper bound for the case that the running time exceeded the time bound, we measured the difference between the upper bound and the optimal value for all examples. Table \ref{Mdense} reports the maximum of the differences for all examples corresponding to the solvers.  \\
\begin{table*}[t]
\caption{Dense instances}
\label{dense}
\begin{tabular}{c c c c c c c c c c}
   \toprule
    \multirow{2}[4]{*}{Instance }  &\multirow{2}[4]{*}{$q^\star$}  & \multicolumn{2}{c}{BARON} & \multicolumn{2}{c}{Couenne} & \multicolumn{2}{c}{CPLEX} & \multicolumn{2}{c}{Algorithm 1}\\
    \cmidrule(rl){3-4}   \cmidrule(rl){5-6}  \cmidrule(rl){7-8}   \cmidrule(rl){9-10}
&         & $lb$ & $time$   & $lb$ & $time$ & $lb$ & $time$ & $lb$ & $time$ \\
  \hline
 \multicolumn{1}{l}{Ex1-40} &  -2286.1  &-7421.5 &  1000 & -2842.8  &   1000  &  -2286.1    & 36 & -2286.1 & 7    \\
 \multicolumn{1}{l}{Ex2-40}& -3821.4    & -13627 & 1000 & -6390.5  &   1000  &  -3821.4   &  323  & -3821.4  & 204  \\
\multicolumn{1}{l}{Ex3-40} &   -2756.6  &-10617 & 1000 & -4522.3  &   1000  &  -2756.6   &  100  &  -2756.6 & 11   \\
 \multicolumn{1}{l}{Ex4-40}&   -2341.6    & -5714  & 1000 & -4287  &   1000  &    -2341.6   &  39  &  -2341.6   & 9  \\
\multicolumn{1}{l}{Ex5-40}&  -2808.2    & -5660  & 1000 & -4195.5  &   1000  & -2808.2   &  124  & -2808.2  & 11 \\
\multicolumn{1}{l}{Ex6-40}& -4341.8     & -25538  & 1000 & -5319.4  &   1000  & -4341.8    &  30  & -4341.8    &4 \\
\multicolumn{1}{l}{Ex7-40}&  -2465.4    & -5916.9   & 1000 & -3326.3  &   1000  &   -2465.4   &  91  & -2465.4   & 3  \\
\multicolumn{1}{l}{Ex8-40}& -2554.6     & -6570.7  & 1000 & -5246.1  &   1000  &  -2554.6   &  564  &-2554.6 & 424   \\
\multicolumn{1}{l}{Ex9-40}&  -4599.6    & -16653   & 1000 & -5381.7  &   1000  &  -4599.6   &  26  & -4599.6  & 3  \\
\multicolumn{1}{l}{Ex10-40} &   -3446.6  &  -8798.8    & 1000 & -4835.9  &   1000  &  -3446.6    &  40 & -3446.6    & 4\\
 \multicolumn{1}{l}{Ex1-45} &  -4493.2  & -6463.3 &  1000 & -2842.8  &   1000  &  -4493.2  &59 & -4493.2 & 13   \\
\multicolumn{1}{l}{Ex2-45}&  -2705.9  &-13214& 1000 & -6039.8  &   1000  &  -2722.9   & 1000  & -2705.9  & 94  \\
\multicolumn{1}{l}{Ex3-45}  & -3057.8   & -19086 & 1000 & -6271.7  &   1000  &   -3057.8   &  461  & -3057.8 & 196  \\
 \multicolumn{1}{l}{Ex4-45}& -2714.1    & -8607.3 & 1000 & -6092.4  &   1000  &  -2714.1   &  689  & -2714.1  & 698  \\
\multicolumn{1}{l}{Ex5-45} &   -3028.2   & -13511  & 1000 & -6822.5  &   1000  &  -3075.1   &  1000  & -3028.2  & 888  \\
\multicolumn{1}{l}{Ex6-45} &   -2354.4   & -10756  & 1000 & -6215.4  &   1000  &  -2549.2   &  1000  & -2354.4   & 657 \\
\multicolumn{1}{l}{Ex7-45} & -3391.4    & -15958  & 1000 & -6561  &   1000  &  -3391.4   &  197  & -3391.4   & 96  \\
\multicolumn{1}{l}{Ex8-45}&  -1948.2    & -6923.2 & 1000 & -3675.8  &   1000 & -1948.2    &  341  & -1948.2  & 9   \\
\multicolumn{1}{l}{Ex9-45} &  -2710.2   & -8781.2   & 1000 & -5014  &  1000  &  -2710.2  &  172  & -2710.2  & 35  \\
\multicolumn{1}{l}{Ex10-45}& -3099     &  -8431.3   & 1000 & -6239.7  &   1000  &  -3099    & 193 & -3099   & 91\\

    \midrule
    \bottomrule
\end{tabular}
\end{table*}
\begin{table*}[t]
\caption{Dense instances}
\label{Mdense}
\begin{tabular}{c c c c c c c c c c}
   \toprule
& \multicolumn{1}{c}{BARON} & \multicolumn{1}{c}{Couenne} & \multicolumn{1}{c}{CPLEX} & \multicolumn{1}{c}{Algorithm 1}\\
  \hline
& 2.1    &  180   & 50 & 0\\
    \midrule
    \bottomrule
\end{tabular}
\end{table*}
\indent
For the third group of the instances, we regarded concave QPs with sparsity. We selected twenty examples form RandQP folder in \cite{Chen}. In most of the instances, $Q$ were indefinite. We shifted eigenvalues such that $Q$ was transformed to a negative semi-definite matrix. Moreover, we considered the instances without equality constraints. As there were box constraints in all instances, the feasible set was bounded. Table \ref{sparse} summarizes the computational performances. For this group of test problems, all solvers gave the upper bound equal to the optimal value.\\
\begin{table*}[t]
\caption{Sparse instances}
\label{sparse}
\begin{tabular}{c c c c c c c c c c}
   \toprule
    \multirow{2}[4]{*}{Instance }  &\multirow{2}[4]{*}{$q^\star$}  & \multicolumn{2}{c}{BARON} & \multicolumn{2}{c}{Couenne} & \multicolumn{2}{c}{CPLEX} & \multicolumn{2}{c}{Algorithm 1}\\
    \cmidrule(rl){3-4}   \cmidrule(rl){5-6}  \cmidrule(rl){7-8}   \cmidrule(rl){9-10}
&         & $lb$ & $time$   &$lb$ & $time$ & $lb$ & $time$ & $lb$ & $time$ \\
  \hline
    qp40-20-2-1  & -286.31 & -286.31 & 3     & -286.31 & 10    & -286.31 & 1     & -286.31 & 14 \\
    qp40-20-2-2  & -169.572 & -169.572 & 21    & -169.572 & 23    & -169.572 & 1     & -169.572 & 556 \\
    qp40-20-2-3  & -152.31 & -152.31 & 91    & -152.31 & 51    & -152.31 & 4     & -152.31 & 500 \\
    qp40-20-3-1  & -219.664 & -219.664 & 35    & -219.665 & 1000  & -219.664 & 1     & -219.664 & 28 \\
    qp40-20-3-2  & -171.255 & -171.255 & 35    & -171.255 & 43    & -171.255 & 2     & -171.255 & 501 \\
    qp40-20-3-3  & -101.248 & -101.248 & 25    & -101.248 & 39    & -101.248 & 2     & -101.248 & 255 \\
    qp40-20-3-4  & -118.119 & -118.119 & 136   & -118.12 & 1000  & -118.119 & 5     & -118.119 & 53 \\
    qp40-20-4-1  & -240.464 & -240.464 & 188   & -240.464 & 772   & -240.464 & 6     & -240.464 & 104 \\
    qp40-20-4-2  & -168.813 & -168.813 & 117   & -168.813 & 90    & -168.813 & 5     & -168.813 & 28 \\
    qp40-20-4-3  & -93.511 & -93.511 & 662   & -93.511 & 337   & -93.511 & 38    & -93.511 & 561 \\
    qp50-25-1-1  & -430.892 & -430.892 & 49    & -430.892 & 518   & -430.892 & 3     & -430.892 & 557 \\
    qp50-25-1-2  & -131.88 & -131.88 & 71    & -131.88 & 189   & -131.88 & 8     & -131.88 & 619 \\
    qp50-25-1-3  & -137.567 & -137.567 & 167   & -137.569 & 1000  & -137.567 & 7     & -137.567 & 799 \\
    qp50-25-1-4  & -133.52 & -134.151 & 1000  & -133.521 & 1000  & -133.52 & 10    & -133.52 & 716 \\
    qp50-25-2-1  & -269.924 & -269.924 & 84    & -269.924 & 136   & -269.924 & 7     & -269.924 & 648 \\
    qp50-25-2-2  & -204.733 & -204.733 & 654   & -204.733 & 369   & -204.733 & 28    & -204.733 & 612 \\
    qp50-25-2-3  & -167.34 & -167.341 & 1000  & -167.34 & 934   & -167.34 & 8     & -167.34 & 8 \\
    qp50-25-2-4  & -129.209 & -129.21 & 1000  & -129.209 & 107   & -129.209 & 7     & -129.209 & 615 \\
    qp50-25-3-1  & -393.76 & -393.76 & 40    & -393.761 & 1000  & -393.76 & 2     & -393.76 & 40 \\
    qp50-25-3-2 & -224.266 & -224.268 & 1000  & -224.267 & 1000  & -224.266 & 14    & -224.266 & 23 \\

    \midrule
    \bottomrule
\end{tabular}
\end{table*}

\indent
For last group of instances, we considered  the norm maximization problem. This problem can be formulated as concave QP
\begin{equation*}
\begin{array}{ll}
 \ & \min \ -x^Tx
\\ &  s.t. \ Ax\leq b,
\end{array}
\end{equation*}
where $X=\{x: Ax\leq b\}$ is a polytope. Unlike  norm maximization problem, the above problem is  NP-hard. To evaluate the performance of the solvers, we considered the polytopes which were generated for the second group.   Tables \ref{Max} and \ref{MMax} give computational performances.
\begin{table*}[t]
\caption{Max norm instances}
\label{Max}
\begin{tabular}{c c c c c c c c c c}
   \toprule
    \multirow{2}[4]{*}{Instance }  &\multirow{2}[4]{*}{$q^\star$}  & \multicolumn{2}{c}{BARON} & \multicolumn{2}{c}{Couenne} & \multicolumn{2}{c}{CPLEX} & \multicolumn{2}{c}{Algorithm 1}\\
    \cmidrule(rl){3-4}   \cmidrule(rl){5-6}  \cmidrule(rl){7-8}   \cmidrule(rl){9-10}
&         & $lb$ & $time$   & $lb$ & $time$ & $lb$ & $time$ & $lb$ & $time$ \\
  \hline
 \multicolumn{1}{l}{Ex1-40} &  -1087.4  &  -1087.4 &  360 &  -1087.4  &   8  &  -1087.4   & 45 & -1087.4 & 4   \\

 \multicolumn{1}{l}{Ex2-40}&  -1425    & -1425 & 221 & -1448.2  &   1000  &  -1425   &  34 & -1425  & 100  \\

\multicolumn{1}{l}{Ex3-40} &   -1514.5 & -1514.5 & 205 & -1514.5  &   618  &   -1514.5   &  7  &   -1514.5 & 15   \\

 \multicolumn{1}{l}{Ex4-40}&  -1324.1   & -1324.1    &  18 &  -1324.1 &   126  &     -1324.1   &  3 &   -1324.1   &4  \\

\multicolumn{1}{l}{Ex5-40}&  -1206.1   &-1206.1  & 214 & 1206.1  &   1000  &-1206.1   &  20  & -1206.1  & 24 \\

\multicolumn{1}{l}{Ex6-40}& -2104.8    & -2104.8  & 4 & -2104.8 &   288  & -2104.8      &  3  & -2104.8      & 3 \\

\multicolumn{1}{l}{Ex7-40}&  -1150.4   & -1150.8   & 1000 & -1204.1  &   1000  &   -1150.4  &  37  & -1150.4   & 36  \\

\multicolumn{1}{l}{Ex8-40}& -1268.6     & -1268.6   & 305 &  -1268.6  &   702  &  -1268.6    &  24  &-1268.6  & 112   \\

\multicolumn{1}{l}{Ex9-40}&  -2090.7    & -2090.7   & 10 &  -2090.7   &   80  &   -2090.7    &  2  &  -2090.7   & 12  \\

\multicolumn{1}{l}{Ex10-40} & -1503.3  &   -1503.3   & 62 & -1503.3   &   527  &  -1503.3     &  3 & -1503.3     & 3\\

 \multicolumn{1}{l}{Ex1-45} &  -2097.9 & -2097.9 &  12 &  -2097.9   &   85  &  -2097.9   & 3 &  -2097.9  & 40   \\

\multicolumn{1}{l}{Ex2-45}&  -1190.4  & -1190.4 & 578 & -1301.6  &   1000  &  -1190.4   & 51  & -1190.4  & 59  \\

\multicolumn{1}{l}{Ex3-45}  & -1636.8   &  -1636.8 & 472 & -1785.1  &   1000  &   -1636.8   &  19  & -1636.8 & 178  \\

 \multicolumn{1}{l}{Ex4-45}& -1527.3    & -1527.3   & 64 & -1527.3  &   537  &  -1527.3   &  7  & -1527.3  & 81  \\
\multicolumn{1}{l}{Ex5-45} &   -1484.8   &  -1484.8  & 803 & -1557.3  &   1000  &   -3.18896   &  64  & -1484.8  & 531  \\
\multicolumn{1}{l}{Ex6-45} &   -1106.1   & -1117.8  & 1000 & -1191.2 &   1000  &  -1106.1  &  65  & -1106.1   & 502 \\

\multicolumn{1}{l}{Ex7-45} & -1489.6    & -1489.6  & 244 & -1519.4  &   1000  &  -1489.6   &  10  & -1489.6  & 37  \\

\multicolumn{1}{l}{Ex8-45}&  -939.6   & -1021.7 & 1000 & -985 &   1000 & -939.6     &  33  & -939.6   & 6   \\

\multicolumn{1}{l}{Ex9-45} &  -1235.2   & -1235.2   & 243 & -1277.7 &  1000  &  -1235.2  &  17  &  -1235.2  & 47  \\

\multicolumn{1}{l}{Ex10-45}&  -1557.8     &   -1557.8     & 102 & -1557.8   &   870  &  -1557.8     & 10 & -1557.8    & 38\\

    \midrule
    \bottomrule
\end{tabular}
\end{table*}
\begin{table*}[t]
\caption{Max norm instances}
\label{MMax}
\begin{tabular}{c c c c c c c c c c}
   \toprule
& \multicolumn{1}{c}{BARON} & \multicolumn{1}{c}{Couenne} & \multicolumn{1}{c}{CPLEX} & \multicolumn{1}{c}{Algorithm 1}\\
  \hline
& 0    &  0.9   & 0 & 0\\
    \midrule
    \bottomrule
\end{tabular}
\end{table*}
\\
\indent
On average, CPLEX outperformed other solvers in most instances. After CPLEX, Algorithm 1 had better performance compared to BARON and Couenne in most instances.  Especially, it had the best performance  on the second group of instances, but  its performances on the third group of instances was not satisfactory. \\

\clearpage
\bibliographystyle{spmpsci}      
\bibliography{template}   

%

\end{document}